\documentclass[12pt]{amsart}
\usepackage[T1]{fontenc}
\usepackage[latin1]{inputenc}
\usepackage{typearea}
\usepackage{geometry}
\usepackage{ulem}
\usepackage{amsmath}
\usepackage{amssymb}
\usepackage{latexsym}
\usepackage{enumerate}
\usepackage{amsthm}
\usepackage[all]{xy}
\usepackage{hhline}
\usepackage{epsf} 
\usepackage{cite}
\numberwithin{equation}{section}

\newtheorem{theorem}{{\sc Theorem}}[section]
\newtheorem{cor}[theorem]{{\sc Corollary}}

\newtheorem{lemma}[theorem]{{\sc Lemma}}

\theoremstyle{remark}
\newtheorem{remark}[theorem]{{\sc Remark}}

\newtheorem{assumption}[theorem]{\sc Assumption}

\theoremstyle{definition}

\newtheorem{example}[theorem]{\sc example}

\newcommand{\R}{\mathbb{R} }
\newcommand{\N}{\mathbb{N} }

\newcommand{\B}{\mathcal{B}}

\newcommand{\D}{\mathcal{D}}
\newcommand{\W}{\mathcal{W}}
\newcommand{\K}{\mathcal{K}}
\newcommand{\Z}{\mathbb{Z}}
\newcommand{\calH}{\mathcal{H}}
\newcommand{\calL}{\mathcal{L}}

\newcommand{\Vtilde}{\tilde{V}}

\providecommand{\abs}[1]{\lvert #1\rvert}
\providecommand{\babs}[1]{\bigl\lvert #1\bigr\rvert}
\providecommand{\Babs}[1]{\Bigl\lvert #1\Bigr\rvert}

\providecommand{\fnorm}[1]{\lVert #1\rVert_\infty}
\providecommand{\pnorm}[1]{\lVert #1\rVert_p}
\providecommand{\einsnorm}[1]{\lVert #1\rVert_1}
\DeclareMathOperator{\Var}{Var}

\DeclareMathOperator{\Poi}{Poisson}
\DeclareMathOperator{\Bin}{Bin}
\DeclareMathOperator{\Hyp}{Hyp}
\DeclareMathOperator{\Bern}{Bernoulli}

\DeclareMathOperator{\Lip}{Lip}

\renewcommand{\phi}{\varphi}
\renewcommand{\epsilon}{\varepsilon}
\newcommand{\eps}{\varepsilon}
\renewcommand{\rho}{\varrho}

\begin{document}
\title[Error bounds in the CLT for random sums]{New Berry-Esseen and Wasserstein bounds in the CLT for non-randomly centered random sums by probabilistic methods}
\author{Christian D\"obler}
\thanks{Universit\'{e} du Luxembourg, Unit\'{e} de Recherche en Math\'{e}matiques \\
christian.doebler@uni.lu}
\begin{abstract}
We prove abstract bounds on the Wasserstein and Kolmogorov distances between non-randomly centered random sums of real i.i.d. random variables with a finite third moment and the standard normal distribution. 
Except for the case of mean zero summands, these bounds involve a coupling of the summation index with its size biased distribution as was previously considered in \cite{GolRin96} for the normal approximation of nonnegative random variables. 
When being specialized to concrete distributions of the summation index like the Binomial, Poisson and Hypergeometric distribution, our bounds turn out to be of the correct order of magnitude.  
\end{abstract}

\maketitle

\section{Introduction}\label{Intro}
Let $N,X_1,X_2,\dotsc$ be random variables on a common probability space such that the $X_j$, $j\geq1$, are real-valued and $N$ assumes values in the set of nonnegative 
integers $\Z_+=\{0,1,\dotsc\}$. Then, the random variable 
\begin{equation}\label{defrs}
S:=\sum_{j=1}^N X_j
\end{equation}
is called a \textit{random sum}. Such random variables appear frequently in modern probabiliy theory, as many models for example from physics, finance, reliability and risk theory naturally lead to the consideration of such sums. Furthermore, sometimes a model, which looks quite different from \eqref{defrs} at the outset, may be transformed into a random sum and then general theory of such sums may be invoked to study the original model \cite{GneKo96}.
For example, by the recent so-called \textit{master Steiner formula} from \cite{McCTr14} the distribution of the metric projection of a standard Gaussian vector onto a closed convex cone in Euclidean space can be represented as a random sum of i.i.d. centered chi-squared random variables with the distribution of $N$ given by the \textit{conic intrinsic volumes} of the cone. Hence, this distribution belongs to the class of the so-called chi-bar-square distributions, which is ubiquitous in the theory of hypotheses testing with inequality constraints (see e.g. \cite{Dyk91} and \cite{Sha88}). This representation was used in \cite{GNP14} to prove quantitative CLTs for both the distribution of the metric projection and the conic intrinsic volume distribution. These results are of interest e.g. in the field of compressed sensing.\\
There already exists a huge body of literature about the asymptotic distributions of random sums. Their investigation evidently began with the work \cite{Rob48} of Robbins, who assumes that the random variables $X_1,X_2,\dotsc$ are i.i.d. with a finite second moment and that $N$ also has a finite second moment. One of the results of \cite{Rob48} is that under these assumptions asymptotic normality of the index $N$ automatically implies asymptotic normality of the corresponding random sum. The book \cite{GneKo96} gives a comprehensive description of the limiting behaviour of such random sums under the assumption that the random variables $N,X_1,X_2,\dotsc$ are independent. In particular, one may ask under what conditions the sum $S$ in \eqref{defrs} is asymptotically normal, where asymptotically refers to the fact that 
the \textit{random index} $N$ in fact usually depends on a parameter, which is sent either to infinity or to zero. Once a CLT is known to hold, one might ask about the accuracy of the normal approximation to the distribution of the given random sum. It turns out that it is generally much easier to derive rates of convergence for random sums of centered random variables, or, which amounts to the same thing, for random sums centered by random variables than for random sums of not necessarily centered random variables. In the centered case one might, for instance, first condition on the value of the index $N$, then use known error bounds for sums of a fixed number of independent random variables like the classical Berry-Esseen theorem and, finally, take expectation with respect to $N$. This technique is illustrated e.g. in the manuscript \cite{Doe12rs} and also works for non-normal limiting distributions like the Laplace distribution. For this reason we will mainly be interested in deriving sharp rates of convergence for the case of non-centered summands, but will also consider the mean-zero case and hint at the relevant differences. Also, we will not assume from the outset that the index $N$ has a 
certain fixed distribution like the Binomial or the Poisson, but will be interested in the general situation.\\
For non-centered summands and general index $N$, the relevant literature on rates of convergence in the random sums CLT seems quite easy to survey. Under the same assumptions as in \cite{Rob48} the paper \cite{Eng83} gives an upper bound on the Kolmogorov distance between the distribution of the random sum and a suitable normal distribution, which is proved to be sharp in some sense. However, this bound is not very explicit as it contains the Kolmogorov distance of $N$ to the normal distribution with the same mean and variance as $N$ as one of the terms appearing in the bound, for instance. This might make it difficult to apply this result to a concrete distribution of $N$. Furthermore, the method of proof cannot be easily adapted to probability metrics different from the Kolmogorov distance like e.g. the Wasserstein distance. In \cite{Kor87} a bound on the Kolmogorov distance is given which improves upon the result of \cite{Eng83} with respect to the constants appearing in the bound. However, the bound given in \cite{Kor87} is no longer strong enough to assure the well-known asymptotic normality of Binomial and Poisson random sums, unless the summands are centered. The paper \cite{Kor88} generalizes the results from \cite{Eng83} to the case of not necessarily identically distributed summands and to situations, where the summands might not have finite absolute third moments. However, at least for non-centered summands, the bounds in \cite{Kor88} still lack some explicitness. \\
To the best of our knowledge, the article \cite{Sunk13} is the only one, which gives bounds on the Wasserstein distance between random sums for general indices $N$ and the standard normal distribution. However, as mentioned by the same author in \cite{Sunk14}, the results of \cite{Sunk13} generally do not yield accurate bounds, unless the summands are centered. Indeed, the results from \cite{Sunk13} do not even yield convergence in distribution for Binomial or Poisson random sums of non-centered summands.\\ 
The main purpose of the present article is to combine Stein's method of normal approximation with several modern probabilistic concepts like certain coupling constructions and conditional independence, to prove accurate abstract upper bounds on the distance between suitably standardized random sums of i.i.d. summands measured by two popular probability metrics, the Kolmogorov and Wasserstein distances. Using a simple inequality, this gives bounds for the whole classe of $L^p$ distances of distributions, $1\leq p\leq\infty$. These upper bounds, in their most abstract forms (see Theorem \ref{maintheo} below), 
involve moments of the difference of a coupling of $N$ with its \textit{size-biased distribution} but reduce to very explicit expressions if either $N$ has a concrete distribution like the Binomial, Poisson or dirac delta distribution, the summands $X_j$ are centered, or, if the distribution of $N$ is infinitely divisible. These special cases are extensively presented in order to illustrate the wide applicability and strength of our results. As indicated above, this seems to be the first work which gives Wasserstein bounds in the random sums CLT for general indices $N$, which reduce to bounds of optimal order, when being specialized to concrete distributions like the Binomial and the Poisson distributions. Using our abstract approach via size-bias couplings, we are also able to prove rates for Hypergeometric random sums. These do not seem to have been treated in the literature, yet. This is not a surprise, because the Hypergeometric distribution is conceptually more complicated than the Binomial or Poisson distribution, as it is neither a natural convolution of i.i.d. random variables nor infinitely divisible. Indeed, every distribution of the summation index which allows for a close size-bias coupling should be amenable to our approach.\\ 
It should be mentioned that Stein's method and coupling techniques have previously been used to bound the error of exponential approximation \cite{PekRol11} and approximation by the Laplace distribution \cite{PiRen12} of certain random sums. In these papers, the authors make use of the fact that the exponential distribution and the Laplace distribution are the unique fixed points of certain distributional transformations and are able to succesfully couple the given random sum with a random variable having the respective transformed distribution. In the case of the standard normal distribution, which is a fixed point of the zero-bias transformation from \cite{GolRei97}, it appears 
natural to try to construct a close coupling with the zero biased distribution of the random sum under consideration. However, interestingly it turns out that we are only able to do so in the case of centered summands whereas for the general case an intermediate step involving a coupling of the index $N$ with its size biased distribution is required for the proof. Nevertheless, the zero-bias transformation or rather an extension of it to non-centered random variables, plays an important role for our argument. This combination of two coupling constructions which belong to the classical tools of Stein's method for normal approximation is a new feature lying at the heart of our approach. \\
The remainder of the article is structured as follows: In Section \ref{results} we review the relevant probability distances, the size biased distribution and state our quantitative results on the normal approximation of random sums. 
Furthermore, we prove new identities for the distance of a nonnegative random variable to its size-biased distribution in three prominent metrics and show that for some concrete distributions, natural couplings are $L^1$-optimal and, hence, yield the Wasserstein distance. In Section \ref{stein} we collect necessary facts from Stein's method of normal approximation and introduce a variant of the zero-bias transformation, which we need for the proofs of our results. Then, in Section \ref{proof}, the proof of our main theorems, Theorem \ref{maintheo} and Theorem \ref{meanzero} is given. Finally, Section \ref{Appendix} contains the proofs of some auxiliary results, needed for the proof of the Berry-Esseen bounds in Section \ref{proof}.

\section{Main results}\label{results}
Recall that for probability measures $\mu$ and $\nu$ on $(\R,\B(\R))$, their \textit{Kolmogorov distance} is defined by
\begin{equation*}
d_\K(\mu,\nu):=\sup_{z\in\R}\babs{\mu\bigl((-\infty,z]\bigr)- \mu\bigl((-\infty,z]\bigr) }=\fnorm{F-G}\,,
\end{equation*}
where $F$ and $G$ are the distribution functions corresponding to $\mu$ and $\nu$, respectively. Also, if both $\mu$ and $\nu$ have finite first absolute moment, then one defines the \textit{Wasserstein distance} between them via 
\begin{equation*}
d_\W(\mu,\nu):=\sup_{h\in\Lip(1)}\Babs{\int h d\mu -\int h d\nu}\,, 
\end{equation*}
where $\Lip(1)$ denotes the class of all Lipschitz-continous functions $g$ on $\R$ with Lipschitz constant not greater than $1$. In view of Lemma \ref{distlemma} below, we also introduce the \textit{total variation distance} bewtween 
$\mu$ and $\nu$ by
\begin{equation*}
d_{TV}(\mu,\nu):=\sup_{B\in\B(\R)}\babs{\mu(B)-\nu(B)}\,.
\end{equation*}
If the real-valued random variables $X$ and $Y$ have distributions $\mu$ and $\nu$, respectively, then 
we simply write $d_\K(X,Y)$ for $d_\K\bigl(\calL(X),\calL(Y)\bigr)$ and similarly for the Wasserstein and total variation distances and also speak of the respective distance between the random variables $X$ and $Y$. 
Before stating our results, we have to review 
the concept of the size-biased distribution corresponding to a distribution supported on $[0,\infty)$. 
Thus, if $X$ is a nonnegative random variable with $0<E[X]<\infty$, then a random variable $X^s$ is said to have the 
\textit{$X$-size biased distribution}, if for all bounded and measurable functions $h$ on $[0,\infty)$ 
\begin{equation}\label{sbdef}
E[Xh(X)]=E[X]E[h(X^s)]\,,
\end{equation}   
see, e.g. \cite{GolRin96}, \cite{ArrGol} or \cite{AGK}. Equivalently, the distribution of $X^s$ has Radon-Nikodym derivative with respect to the distribution of $X$ given by
\begin{equation*}
\frac{P(X^s\in dx)}{P(X\in dx)}=\frac{x}{E[X]}\,,
\end{equation*}
which immediately implies both existence and uniqueness of the $X$-size biased distribution. Also note that \eqref{sbdef} holds true for all measurable functions 
$h$ for which $E\abs{Xh(X)}<\infty$. 
In consequence, if $X\in L^p(P)$ for some $1\leq p<\infty$, then $X^s\in L^{p-1}(P)$ and 
\begin{equation*}
 E\Bigl[\bigl(X^s\bigr)^{p-1}\Bigr]=\frac{E\big[X^p\bigr]}{E[X]}\,.
\end{equation*}

The following lemma, which seems to be new and might be of independent interest, gives identities for the distance of $X$ to $X^s$ in the three metrics mentioned above. The proof is deferred to the end of this section.
\begin{lemma}\label{distlemma}
Let $X$ be a nonnegative random variable such that $0<E[X]<\infty$. Then, the following identities hold true: 
\begin{enumerate}[{\normalfont (a)}]
\item $\displaystyle d_\K(X,X^s)=d_{TV}(X,X^s)=\frac{E\abs{X-E[X]}}{2E[X]}$.
\item If additionally $E[X^2]<\infty$, then $\displaystyle d_\W(X,X^s)=\frac{\Var(X)}{E[X]}$.
\end{enumerate}
\end{lemma}

\begin{remark}\label{sbrem}
\begin{enumerate}[(a)]
 \item It is well known (see e.g. \cite{Dud02}) that the Wasserstein distance $d_\W(X,Y)$ between the real random variables $X$ and $Y$ has the dual representation 
\begin{equation}\label{dualwas}
d_\W(X,Y)=\inf_{(\hat{X},\hat{Y})\in\pi(X,Y)}E\abs{\hat{X}-\hat{Y}}\,,
\end{equation}
where $\pi(X,Y)$ is the collection of all couplings of $X$ and $Y$, i.e. of all pairs $(\hat{X},\hat{Y})$ of random variables on a joint probability space 
such that $\hat{X}\stackrel{\D}{=}X$ and $\hat{Y}\stackrel{\D}{=}Y$. Also, the infimum in \eqref{dualwas} is always attained, e.g. by the quantile transformation: 
If $U$ is uniformly distributed on $(0,1)$ and if, for a distribution function $F$ on $\R$, we let 
\[F^{-1}(p):=\inf\{x\in\R\,:\,F(x)\geq p\},\quad p\in(0,1)\,,\]
denote the corresponding \textit{generalized inverse} of $F$, then $F^{-1}(U)$ is a random variable with distribution function $F$. Thus, letting 
$F_X$ and $F_Y$ denote the distribution functions of $X$ and $Y$, respectively, it was proved e.g. in \cite{Major78} that 
\[\inf_{(\hat{X},\hat{Y})\in\pi(X,Y)}E\abs{\hat{X}-\hat{Y}}=E\abs{F_X^{-1}(U)-F_Y^{-1}(U)}=\int_0^1\abs{F_X^{-1}(t)-F_Y^{-1}(t)}dt\,.\]
Furthermore, it is not difficult to see that $X^s$ is always stochastically larger than $X$, implying that there is a coupling $(\hat{X},\hat{X^s})$ of $X$ and $X^s$ such that $\hat{X^s}\geq \hat{X}$ (see \cite{ArrGol} for details). 
In fact, this property is already achieved by the coupling via the quantile transformation. By the dual representation \eqref{dualwas} and the fact that the coupling via the quantile transformation 
yields the minimum $L^1$ distance in \eqref{dualwas} we can conclude that \textit{every} coupling $(\hat{X},\hat{X^s})$ such that $\hat{X^s}\geq \hat{X}$ is optimal in this sense, since 
\begin{align*}
 E\babs{\hat{X^s}-\hat{X}}&=E\bigl[\hat{X^s}\bigr]-E\bigl[\hat{X}\bigr]=E\bigl[F_{X^s}^{-1}(U)\bigr]-E\bigl[F_X^{-1}(U)\bigr]\\
 &=E\babs{F_{X^s}^{-1}(U)-F_X^{-1}(U)}=d_\W(X,X^s)\,.
\end{align*}
Note also that, by the last computation and part by (b) of Lemma \ref{distlemma}, we have 
\begin{equation*}
 E\bigl[X^s\bigr]-E\bigl[X\bigr]=E\bigl[\hat{X^s}\bigr]-E\bigl[\hat{X}\bigr]=d_\W(X,X^s)=\frac{\Var(X)}{E[X]}\,.
\end{equation*}

\item Due to a result by Steutel \cite{Steu73}, the distribution of $X$ is infinitely divisible, if and only if there exists a coupling $(X,X^s)$ of $X$ and $X^s$ such that $X^s-X$ is nonnegative and independent of $X$ (see e.g. \cite{ArrGol} for a nice exposition and a proof of this result). 
According to (a) such a coupling always achieves the minimum $L^1$-distance.
\item It might seem curious that according to part (a) of Lemma \ref{distlemma}, the Kolmogorov distance and the total variation distance between a nonnegative random variable and one with its size biased distribution 
always coincide. Indeed, this  holds true since for each Borel-measurable set $B\subseteq \R$ we have the inequality
\begin{align*}
 \babs{P(X^s\in B)-P(X\in B)}&\leq\babs{P(X^s> m)-P(X>m)}\\
 &\leq d_\K(X,X^s)\,,
\end{align*}
where $m:=E[X]$.
Thus, the supremum in the definition 
\begin{equation*}
 d_{TV}(X,X^s)=\sup_{B\in\B(\R)}\babs{P(X^s\in B)-P(X\in B)}
\end{equation*}
of the total variation distance is assumed for the set $B=(m,\infty)$. This can be shortly proved and explained in the following way: For $t\in\R$, using the defining property \eqref{sbdef} of the size biased distribution, 
we can write 
\begin{equation*}
 H(t):=P(X^s\leq t)-P(X\leq t)=m^{-1}E\bigl[(X-m))1_{\{X\leq t\}}\bigr]\,.
\end{equation*}
Thus, for $s<t$ we have  
\begin{equation*}
H(t)-H(s)=m^{-1} E\bigl[(X-m))1_{\{s<X\leq t\}}\bigr]\,,
\end{equation*}
and, hence, $H$ is decreasing on $(-\infty,m)$ and increasing on $(m,\infty)$. Thus, for every Borel set $B\subseteq\R$ we conclude that 
\begin{align*}
 &\; P(X^s\in B)-P(X\in B)=\int_\R 1_B(t)dH(t)\leq \int_\R 1_{B\cap (m,\infty)}(t)dH(t)\\
 &\leq \int_\R 1_{(m,\infty)}(t)dH(t)=P(X^s>m)-P(X>m)\,.
\end{align*}
Note that for this argumentation we heavily relied on the defining property \eqref{sbdef} of the size biased distribution which guaranteed the monotonicity property of the difference $H$ of 
the distribution functions of $X^s$ and $X$, respectively. Since $X^s$ is stochastically larger than $X$, one might suspect that the coincidence of the total variation and the Kolmogorov distance 
holds true in this more general situation. However, observe that the fact that $X^s$ dominates $X$ stochastically only implies that $H\leq0$ but that it is the monotonicity of $H$ on 
$(-\infty,m)$ and on $(m,\infty)$ that was crucial for the derivation. 
\end{enumerate}
\end{remark}

\begin{example}\label{exdist}
\begin{enumerate}[(a)]
\item Let $X\sim\Poi(\lambda)$ have the Poisson distribution with paramter $\lambda>0$. From the Stein characterization of $\Poi(\lambda)$ (see \cite{Ch75}) it is known that 
\begin{equation*}
E[Xf(X)]=\lambda E[f(X+1)]=E[X]E[f(X+1)]
\end{equation*}
for all bounded and measurable $f$. Hence, $X+1$ has the $X$-size biased distribution. As $X+1\geq X$, by Remark \ref{sbrem} this coupling yields the minimum $L^1$-distance between $X$ and $X^s$, which is equal to $1$ in this case. 
\item Let $n$ be a positive integer, $p\in(0,1]$ and let $X_1,\dotsc,X_n$ be i.i.d. random variables such that $X_1\sim\Bern(p)$. 
Then, 
\[X:=\sum_{j=1}^n X_j\sim\Bin(n,p)\]
 has the Binomial distribution with parameters $n$ and $p$.
From the construction in \cite{GolRin96} one easily sees that 
\[X^s:=\sum_{j=2}^{n}X_j+1\]
has the $X$-size biased distribution. As $X^s\geq X$, by Remark \ref{sbrem} this coupling yields the minimum $L^1$-distance between $X$ and $X^s$, which is equal to
\begin{equation*}
 d_\W(X,X^s)=E[1-X_1]=1-p=\frac{\Var(X)}{E[X]}
\end{equation*}
in accordance with Lemma \ref{distlemma}.
\item Let $n,r,s$ be positive integers such that $n\leq r+s$ and let $X\sim\Hyp(n;r,s)$ have the Hypergeometric distribution with parameters $n,r$ and $s$, i.e.
\begin{equation*}
P(X=k)=\frac{\binom{r}{k}\binom{s}{n-k}}{\binom{r+s}{n}}\,,\quad k=0,1,\dotsc,n 
\end{equation*}
with $E[X]=\frac{nr}{r+s}$.
Imagaine an urn with $r$ red and $s$ silver balls. If we draw $n$ times without replacement from this urn and denote by $X$ the total number of drawn red balls, then $X\sim\Hyp(n;r,s)$. For $j=1,\dotsc,n$ denote by $X_j$ the indicator of the event that a red ball is drawn at the $j$-th draw. Then, $X=\sum_{j=1}^n X_j$ and since the $X_j$ are exchangeable, the well-known construction of a random variable $X^s$ wth the $X$-size biased 
distribution from \cite{GolRin96} gives that $X^s=1+\sum_{j=2}^n X_j'$, where 
\begin{equation*}
 \calL\bigl((X_2',\dotsc,X_n')\bigr)=\calL\bigl((X_2,\dotsc,X_n)\,\bigl|\,X_1=1\bigr)\,.
\end{equation*}
But given $X_1=1$ the sum $\sum_{j=2}^n X_j$ has the Hypergeometric distribution with parameters $n-1,r-1$ and $s$ and, hence,  
\begin{equation*}
 X^s\stackrel{\D}{=}Y+1\,,\quad\text{where } Y\sim\Hyp(n-1;r-1,s)\,.
\end{equation*}
In order to construct an $L^1$-optimal coupling of $X$ and $X^s$, fix one of the red balls in the urn and, for $j=2,\dotsc,n$, denote by $Y_j$ the indicator of the event that at the $j$-th draw this 
fixed red ball is drawn. Then, it is not difficult to see that 
\begin{align*}
 Y&:=1_{\{X_1=1\}}\sum_{j=2}^n X_j+1_{\{X_1=0\}}\sum_{j=2}^n(X_j-Y_j)\sim\Hyp(n-1;r-1,s)
\end{align*}
and, hence, 
\begin{align*}
X^s&:=Y+1=1_{\{X_1=1\}}\sum_{j=2}^n X_j+1_{\{X_1=0\}}\sum_{j=2}^n(X_j-Y_j)+1\\
&= 1_{\{X_1=1\}}X+1_{\{X_1=0\}}\Bigl(X+1-\sum_{j=2}^n Y_j\Bigr)
\end{align*}
has the $X$-size biased distribution. Note that since $\sum_{j=2}^n Y_j\leq 1$ we have 
\begin{equation*}
 X^s-X=1_{\{X_1=0\}}\Bigl(1-\sum_{j=2}^n Y_j\Bigr)\geq0\,,
\end{equation*}
and consequently, by Remark \ref{sbrem} (a), the coupling $(X,X^s)$ is optimal in the $L^1$-sense and yields the Wasserstein distance between $X$ and $X^s$:
\begin{equation*}
d_\W(X,X^s)=E\babs{X^s-X}=\frac{\Var(X)}{E[X]}=\frac{n\frac{r}{r+s}\frac{s}{r+s}\frac{r+s-n}{r+s-1}}{\frac{nr}{r+s}}=\frac{s(r+s-n)}{(r+s)(r+s-1)}\,. 
\end{equation*}
\end{enumerate}
\end{example}

We now turn back to the asymptotic behaviour of random sums.
We will rely on the following general assumptions and notation, which we adopt and extend from \cite{Rob48}.
\begin{assumption}\label{genass}
The random variables $N,X_1,X_2,\dotsc$ are independent, $X_1,X_2,\dotsc$ being i.i.d. and such that $E\abs{X_1}^3<\infty$ and $E[N^3]<\infty$. 
Furthermore, we let 
\begin{align*}
\alpha&:=E[N],\quad \beta^2:=E[N^2],\quad\gamma^2:=\Var(N)=\beta^2-\alpha^2,\quad \delta^3:=E[N^3],\\
a&:=E[X_1],\quad b^2:=E[X_1^2],\quad c^2:=\Var(X_1)=b^2-a^2\text{ and } d^3:=E\babs{X_1-E[X_1]}^3.
\end{align*}
\end{assumption}
By Wald' s equation and the Blackwell-Girshick formula, from Assumption \ref{genass} we have
\begin{equation}\label{meanvar}
\mu:=E[S]=\alpha a\quad\text{and}\quad\sigma^2:=\Var(S)=\alpha c^2+a^2\gamma^2\,.
\end{equation} 
The main purpose of this paper is to assess the accuracy of the standard normal approximation to 
the normalized version 
\begin{equation}\label{defw}
W:=\frac{S-\mu}{\sigma}=\frac{S-\alpha a}{\sqrt{\alpha c^2+a\gamma^2}}
\end{equation} 
of $S$ measured by the Kolmogorov and the Wasserstein distance, respectively.
As can be seen from the paper \cite{Rob48}, under the general assumption that 
\[\sigma^2=\alpha c^2+a^2\gamma^2\to\infty\,, \]
there are three typical situations in which $W$ is asymptotically normal, which we will now briefly review.
\begin{enumerate}[1)]
 \item $c\not=0\not=a$ and $\gamma^2=o(\alpha)$ 
 \item $a=0\not=c$ and $\gamma=o(\alpha)$
 \item $N$ itself is asymptotically normal and at least one of $a$ and $c$ is different from zero.
\end{enumerate}
We remark that 1) roughly means that $N$ tends to infinity in a certain sense, but such that it only fluctuates slightly around its mean $\alpha$ and, thus, behaves more or less as the constant $\alpha$ (tending to infinity). 
 If $c=0$ and $a\not=0$, then we have
 \begin{equation*}
  S=aN\quad\text{a.s.}
 \end{equation*}
and asymptotic normality of $S$ is equivalent to that of $N$. For this reason, unless specifically stated otherwise, we will from now on assume that $c\not=0$. However, we would like to remark that all bounds in which 
$c$ does not appear in the denominator also hold true in the case $c=0$.

\begin{theorem}\label{maintheo}
Suppose that Assumption \ref{genass} holds, let $W$ be given by \eqref{defw} and let $Z$ have the standard normal distribution. Also, let $(N,N^s)$ be a coupling of $N$ and $N^s$ having the $N$-size biased distribution
such that $N^s$ is also independent of $X_1,X_2,\dotsc$ and define $D:=N^s-N$. Then, we have the following bound: 
\begin{enumerate}[{\normalfont (a)}]
 \item $\displaystyle d_\W(W,Z)\leq\frac{2c^2b\gamma^2}{\sigma^3}+\frac{3\alpha d^3}{\sigma^3}+\frac{\alpha a^2}{\sigma^2}\sqrt{\frac{2}{\pi}}\sqrt{\Var\bigl(E[D\,|\,N]\bigr)}\\
 {}\hspace{2cm}+\frac{2\alpha a^2b}{\sigma^3}E\bigl[1_{\{D<0\}}D^2\bigr]+\frac{\alpha\abs{a}b^2}{\sigma^3}E[D^2]$
\item If, additionally, $D\geq0$, then we also have
 \begin{align*}
d_\K(W,Z)&\leq\frac{(\sqrt{2\pi}+4)bc^2\alpha}{4\sigma^3}\sqrt{E[D^2]}+\frac{ d^3\alpha(3\sqrt{2\pi}+4)}{8\sigma^3}+ \frac{c^3\alpha}{\sigma^3}\notag\\
&\;+\Bigl(\frac{7}{2}\sqrt{2}+2\Bigr)\frac{\sqrt{\alpha} d^3}{c\sigma^2}+\frac{c^2\alpha}{\sigma^2}P(N=0)+\frac{ d^3\alpha}{c\sigma^2}E\bigl[N^{-1/2}1_{\{N\geq1\}}\bigr]\notag\\
&\;+\frac{\alpha a^2}{\sigma^2}\sqrt{\Var\bigl(E[D\,|\,N]\bigr)}
+\frac{\alpha \abs{a}b^2}{2\sigma^3}\sqrt{E\Bigl[\bigl(E\bigl[D^2\,\bigl|\,N\bigr]\bigr)^2\Bigr]}\notag\\
&\;+\frac{\alpha \abs{a}b^2\sqrt{2\pi}}{8\sigma^3}E[D^2]
+\frac{\alpha \abs{a}b}{\sigma^2}\sqrt{P(N=0)}\sqrt{E[D^2]}\notag\\
&\;+\frac{\alpha \abs{a}b^2}{c\sigma^2\sqrt{2\pi}}E\bigl[D^21_{\{N\geq1\}}N^{-1/2}\bigr]+\Bigl(\frac{ d^3\alpha \abs{a} b}{\sigma^2}+ \frac{\alpha bc}{\sigma^2\sqrt{2\pi}}\Bigr)E\bigl[D1_{\{N\geq1\}}N^{-1/2}\bigr]\,.
\end{align*}
\end{enumerate}
\end{theorem}

\begin{remark}\label{mtrem}
\begin{enumerate}[(a)]
\item In many concrete situations, one has that a natural coupling of $N$ and $N^s$ yields $D\geq0$ and, hence, Theorem \ref{maintheo} gives bounds on both the Wasserstein and Kolmogorov distances (note that the fourth 
summand in the bound on $d_\W(W,Z)$ vanishes if $D\geq0$). 
For instance, by Remark \ref{sbrem} (b), this is the case, if 
the distribution of $N$ is infinitely divisible. In this case, the random variables $D$ and $N$ can be chosen to be independent and, thus, our bounds can further be simplified (see Corollary \ref{infdiv} below). 
Indeed, since $N^s$ is always stochastically larger than $N$, by Remark \ref{sbrem} (a) it is always possible to construct a coupling $(N,N^s)$ such that $D=N^s-N\geq0$.
\item However, although we know that a coupling of $N$ and $N^s$ such that $D=N^s-N\geq0$ is always possible in principle, sometimes one would prefer working with a feasible and natural coupling which does not have this property. 
For instance, this is the case in the situation of Corollary \ref{divisible} below. This is why we have not restricted ourselves to the case $D\geq0$ but allow for arbitrary couplings $(N,N^s)$. We mention that we also have a bound on the Kolmogorov distance between $W$ and a standard normally distributed $Z$ in this more general situation, which is given by 
\begin{align*}
d_\K(W,Z)&\leq\sum_{j=1}^7 B_j\,,
\end{align*}
where $B_1, B_2, B_4, B_5, B_6$ and $B_7$ are defined in \eqref{e1z7}, \eqref{e12gen}, \eqref{e222}, \eqref{e221b}, \eqref{boundr1} and \eqref{boundr2}, respectively, and 
\begin{equation*}
B_3:=\frac{\alpha a^2}{\sigma^2}\sqrt{\Var\bigl(E[D\,|\,N]\bigr)}\,.
\end{equation*}
It is this bound what is actually proved in Section \ref{proof}. Since it is given by a rather long expression in the most general case, we have decided, however, not to present it within Theorem \ref{maintheo}.
\item We mention that the our proof of the Wasserstein bounds given in Theorem \ref{maintheo} is only roughly five pages long and is not at all technical but rather makes use of probabilistic ideas and concepts. 
The extended length of our derivation is simply due to our ambition to present Kolmogorov bounds as well which, as usual within Stein's method, demand much more technicality.
\end{enumerate}
\end{remark}

The next theorem treats the special case of centered summands.
\begin{theorem}\label{meanzero}
Suppose that Assumption \ref{genass} holds with $a=E[X_1]=0$, let $W$ be given by \eqref{defw} and let $Z$ have the standard normal distribution. Then, 
\begin{align*}
d_\W(W,Z)&\leq\frac{2\gamma}{\alpha}+\frac{3 d^3}{c^3\sqrt{\alpha}} \quad\text{and}\\
d_\K(W,Z)&\leq\frac{(\sqrt{2\pi}+4)\gamma}{4\alpha}+\Bigl(\frac{ d^3(3\sqrt{2\pi}+4)}{8c^3}+ 1\Bigr)\frac{1}{\sqrt{\alpha}} 
+\Bigl(\frac{7}{2}\sqrt{2}+2\Bigr)\frac{ d^3}{c^3\alpha}\notag\\
&\;+P(N=0)+\biggl(\frac{ d^3}{c^3}+
\frac{\gamma}{\sqrt{\alpha}\sqrt{2\pi}}\biggr)\sqrt{E\bigl[1_{\{N\geq1\}}N^{-1}\bigr]}\,.
\end{align*}
\end{theorem}
\begin{remark}\label{remmz}
\begin{enumerate}[(a)]
\item The proof will show that Theorem \ref{meanzero} holds as long as \\
$E[N^2]<\infty$. Thus, Assumption \ref{genass} could be slightly relaxed in this case. 
\item Theorem \ref{meanzero} is not a direct consequence of Theorem \ref{maintheo} as it is stated above. Actually, instead of Theorem \ref{maintheo} we could state a result, 
which would reduce to Theorem \ref{meanzero} if $a=0$, but the resulting bounds would look more cumbersome in the general case. Also, they would be of the same order 
as the bounds presented in Theorem \ref{maintheo} in the case that $a\not=0$. This is why we have refrained from presenting these bounds in the general case but have chosen to prove Theorem \ref{maintheo} and Theorem \ref{meanzero} in parallel.
Note that, if $a\not=0$, then a necessary condition for our bounds to imply the CLT is that 
 \begin{equation}\label{neccond}
  \frac{\alpha}{\sigma^3}E[D^2]=o(1)\quad\text{and}\quad\frac{\alpha}{\sigma^2}\sqrt{\Var\bigl(E[D|N]\bigr)}=o(1)\,.
 \end{equation}
This should be compared to the conditions which imply asymptotic normality for $N$ by size-bias couplings given in \cite{GolRin96}, namely 
\begin{equation}\label{condGolRin}
  \frac{\alpha}{\gamma^3}E[D^2]=o(1)\quad\text{and}\quad\frac{\alpha}{\gamma^2}\sqrt{\Var\bigl(E[D|N]\bigr)}=o(1)\,.
\end{equation}
If \eqref{condGolRin} holds, then from \cite{GolRin96} we know that $N$ is asymptotically normal and, as was shown within the proof of Lemma 1 in \cite{Rob48}, this implies that $\gamma=o(\alpha)$. Since, if $a\not=0$, 
\eqref{condGolRin} implies \eqref{neccond}, we can conclude from Theorems \ref{maintheo} and \ref{meanzero} that $W$ is asymptotically normal. In a nutshell, if the bounds from \cite{GolRin96} on the distance to normality 
of $N$ tend to zero, then so do our bounds and, hence, yield the CLT for $W$. However, the validity of \eqref{condGolRin} is neither necessary for \eqref{neccond} to hold nor for our bounds to imply asymptotic normality 
of $W$ (see Remark \ref{hyprem} (b) below).
\item For distribution functions $F$ and $G$ on $\R$ and $1\leq p<\infty$, one defines their $L^p$-distance by 
\begin{equation*}
\pnorm{F-G}:=\biggl(\int_\R\babs{F(x)-G(x)}^p dx\biggr)^{1/p}\,.
\end{equation*}
It is known (see \cite{Dud02}) that $\einsnorm{F-G}$ coincides with the Wasserstein distance of the corresponding distributions $\mu$ and $\nu$, say. By H\"older's inequality, for $1\leq p<\infty$, we have 
\begin{equation*}
 \pnorm{F-G}\leq d_\K(\mu,\nu)^{\frac{p-1}{p}}\cdot d_\W(\mu,\nu)^{\frac{1}{p}}\,.
\end{equation*}
Thus, our results immediately yield bounds on the $L^p$-distances of $\calL(W)$ and $N(0,1)$.
\item It would be possible to drop the assumption that the summands be identically distributed. For reasons of clarity of the presentation, we have, however, decided to stick to the i.i.d. setting. 
See also the discussion of possible generalizations before the proof of Lemma \ref{distlemma} at the end of this section.
\end{enumerate}
\end{remark}

\begin{cor}\label{infdiv}
Suppose that Assumption \ref{genass} holds, let $W$ be given by \eqref{defw} and let $Z$ have the standard normal distribution. Furthermore, assume that the distribution of the index $N$ is infinitely divisible. Then, we have 
\begin{align*}
&d_\W(W,Z)\leq \frac{2c^2b\gamma^2+3\alpha d^3}{\sigma^3}
+\frac{(\alpha\delta^3-\alpha^2\gamma^2+\gamma^4-\beta^4)\abs{a}b^2}{\alpha\sigma^3}\quad\text{and}\\
&d_\K(W,Z)\leq\frac{ d^3\alpha(3\sqrt{2\pi}+4)}{8\sigma^3}+ \frac{c^3\alpha}{\sigma^3}+\biggl(\frac{7}{2}\sqrt{2}+2\biggr)\frac{\sqrt{\alpha} d^3}{c\sigma^2}+\frac{c^2\alpha}{\sigma^2}P(N=0)\\
&\;+\frac{\abs{a}b^2(\delta^3\alpha+\gamma^4-\beta^4-\gamma^2\alpha^2)}{\alpha\sigma^3}\biggl(\frac{\sqrt{2\pi}}{8}+\frac{1}{2}\biggr)\\
&\;+\sqrt{\delta^3\alpha+\gamma^4-\beta^4-\gamma^2\alpha^2}\biggl(\frac{(\sqrt{2\pi}+4)bc^2}{4\sigma^3}+\sqrt{P(N=0)}\frac{\abs{a}b}{\sigma^2}\biggr)\\
&\;+E\bigl[1_{\{N\geq1\}}N^{-1/2}\bigr]\biggl(\frac{\abs{a}b^2(\delta^3\alpha+\gamma^4-\beta^4-\gamma^2\alpha^2)}{c\alpha\sigma^2\sqrt{2\pi}}+\frac{\gamma^2 d^3\abs{a}b}{\sigma^2}+\frac{ d^3\alpha}{c\sigma^2}
+\frac{\gamma^2bc}{\sigma^2\sqrt{2\pi}}\biggr)
\end{align*}
\end{cor}

\begin{proof}
By Remark \ref{sbrem} (b) we can choose $D\geq0$ independent of $N$ such that $N^s=N+D$ has the 
$N$-size biased distribution. Thus, by independence we obtain
\begin{align*}
\Var(D)&=\Var(N^s)-\Var(N)=E\bigl[(N^s)^2\bigr]-E[N^s]^2-\gamma^2\\
&=\frac{E[N^3]}{E[N]}-\left(\frac{E[N^2]}{E[N]}\right)^2-\gamma^2\\
&=\frac{\delta^3}{\alpha}-\frac{\beta^4}{\alpha^2}-\gamma^2\,.
\end{align*}
This gives 
\[E[D^2]=\Var(D)+E[D]^2=\frac{\delta^3}{\alpha}+\frac{\gamma^4-\beta^4}{\alpha^2}-\gamma^2\,.\]
Also, 
\[\Var\bigl(E[D|N]\bigr)=\Var\bigl(E[D]\bigr)=0\quad\text{and}\quad \sqrt{E\Bigl[\bigl(E\bigl[D^2\,\bigl|\,N\bigr]\bigr)^2\Bigr]}=E[D^2]  \]
in this case. Now, the claim follows from Theorem \ref{maintheo}.\\
\end{proof}

In the case that $N$ is constant, the results from Theorem \ref{maintheo} reduce to the known optimal convergence rates for sums of i.i.d. random variables with finite third moment, albeit with non-optimal constants 
(see e.g. \cite{Shev11} and \cite{Gol10} for comparison).

\begin{cor}\label{ncons}
Suppose that Assumption \ref{genass} holds, let $W$ be given by \eqref{defw} and let $Z$ have the standard normal distribution. Also, assume that the index $N$ is a positive constant. Then, 
\begin{align*}
d_\W(W,Z)&\leq\frac{3 d^3}{c^3\sqrt{N}}\quad\text{and}\\
d_\K(W,Z)&\leq \frac{1}{\sqrt{N}}\biggl(1+\Bigl(\frac{7}{2}\bigl(1+\sqrt{2}\bigr)+\frac{3\sqrt{2\pi}}{8}\Bigr)\frac{ d^3}{c^3}\biggr)\,.
\end{align*}
\end{cor}

\begin{proof}
 In this case, we can choose $N^s=N$ yielding $D=0$ and the result follows from Theorem \ref{maintheo}.\\
\end{proof}

Another typical situation when the distribution of $W$ may be well approximated by the normal is if the index $N$ is itself a sum of many i.i.d. variables. Our results yield very explicit convergence rates in this special case. 
This will be exemplified for the Wasserstein distance by the next corollary. Using the bound presented in Remark \ref{mtrem} (b) one would get a bound on the Kolmogorov distance, which is more complicated but of the same 
order of magnitude. A different way to prove bounds for the CLT by Stein's method in this special situation is presented in Theorem 10.6 of \cite{CGS}. Their method relies on a general bound for the error of normal approximation 
to the distribution of a non-linear statistic of independent random variables which can be written as a linear statistic plus a small remainder term as well as on truncation and conditioning on $N$ in order to apply the classical 
Berry-Esseen theorem. Though our method also makes use of conditioning on $N$, it is more directly tied to random sums and also relies on (variations of) classical couplings in Stein's method
(see the proof in Section \ref{proof} for details).

\begin{cor}\label{divisible}
Suppose that Assumption \ref{genass} holds, let $W$ be given by \eqref{defw} and let $Z$ have the standard normal distribution. Additionally, assume that the distribution of the index $N$ is such that $N\stackrel{\D}{=}N_1+\ldots+N_n$, where $n\in\N$ and $N_1,\dotsc,N_n$ 
are i.i.d. nonnegative random variables such that $E[N_1^3]<\infty$. Then, using the notation 
\begin{align*}
\alpha_1&:=E[N_1]\,,\quad \beta_1^2:=E[N_1^2]\,,\quad\gamma_1^2:=\Var(N_1)\,,\quad\delta_1^3:=E[N_1^3] \quad\text{and}\\
\sigma_1^2&:=c^2\alpha_1 +a^2\gamma_1^2
\end{align*}
we have 
\begin{align*}
d_\W(W,Z)&\leq\frac{1}{\sqrt{n}}\biggl(\frac{2c^2b\gamma_1^2}{\sigma_1^3}+\frac{3\alpha_1 d^3}{\sigma_1^3}+\sqrt{\frac{2}{\pi}}\frac{\alpha_1 a^2\gamma_1^2}{\sigma_1^2}
+\frac{2\alpha_1 (a^2b+\abs{a}b^2)}{\sigma_1^3}\Bigl(\frac{\delta_1^3}{\alpha_1}-\beta_1^2\Bigr)\biggr)\,.
\end{align*}
\end{cor}

\begin{proof}
From \cite{GolRin96} (see also \cite{CGS}) it is known that letting $N_1^s$ be independent of $N_1,\dotsc,N_n$ and have the $N_1$-size biased distribution, a random variable with the $N$-size biased distribution is given by 
\[N^s:= N_1^s+\sum_{j=2}^nN_j\,,\quad\text{yielding}\quad D=N_1^s-N_1\,.\]
Thus, by independence and since $N_1,\dotsc,N_n$ are i.i.d., we have 
\[E[D|N]=E[N_1^s]-\frac{1}{n}N\]
and, hence,
\[\Var\bigl(E[D|N]\bigr)=\frac{\Var(N)}{n^2}=\frac{\gamma_1^2}{n}\,.\]
Clearly, we have 
\begin{equation*}
 \alpha=n\alpha_1\,,\quad \gamma^2=n\gamma_1^2\quad\text{and}\quad\sigma^2=n\sigma_1^2\,.
\end{equation*}
Also, using independence and \eqref{sbdef}, 
\begin{align*}
 E[D^2]&=E\bigl[N_1^2-2N_1N_1^s+(N_1^s)^2\bigr]=\beta_1^2-2\alpha_1E[N_1^s]+E\bigl[(N_1^s)^2\bigr]\notag\\
 &=\beta_1^2-2\alpha_1\frac{\beta_1^2}{\alpha_1}+\frac{\delta_1^3}{\alpha_1}=\frac{\delta_1^3}{\alpha_1}-\beta_1^2\,.
\end{align*}
Thus, the bound follows from Theorem \ref{maintheo}.\\
\end{proof}

Very prominent examples of random sums, which are known to be  asymptotically normal, are Poisson and Binomial random sums. The respective bounds, which follow from our abstract findings, are presented in the next two corollaries.

\begin{cor}\label{Poisson}
Suppose that Assumption \ref{genass} holds, let $W$ be given by \eqref{defw} and let $Z$ have the standard normal distribution. Assume further that $N\sim\Poi(\lambda)$ has the Poisson distribution with parameter $\lambda>0$. Then, 
\begin{align*}
d_\W(W,Z)&\leq\frac{1}{\sqrt{\lambda}}\Bigl(\frac{2c^2}{b^2}+\frac{3 d^3}{b^3}+\frac{\abs{a}}{b}\Bigr)\quad\text{and}\\
d_\K(W,Z)&\leq\frac{1}{\sqrt{\lambda}}\biggl(\frac{\sqrt{2\pi}}{4}+1+\frac{(3\sqrt{2\pi}+4) d^3}{8b^3}+\frac{c^3}{b^3}+\Bigl(\frac{7}{2}\sqrt{2}+3\Bigr)\frac{ d^3}{cb^2}\\
&\qquad+\frac{\abs{a}(\sqrt{2\pi}+4+8 d^3)}{8b}+\frac{\abs{a}}{c\sqrt{2\pi}}+\frac{c}{b\sqrt{2\pi}}\biggr)+\frac{c^2}{b^2}e^{-\lambda}+\frac{\abs{a}}{b}e^{-\lambda/2}\,.
\end{align*}
\end{cor}
\begin{proof}
In this case, by Example \ref{exdist} (a), we can choose $D=1$, yielding that 
\[E[D^2]=1\quad\text{and}\quad\Var\bigl(E[D|N]\bigr)=0\,.\]
Note that 
\begin{equation*}
 E\bigl[1_{\{N\geq1\}}N^{-1/2}\bigr]\leq\sqrt{E\bigl[1_{\{N\geq1\}}N^{-1}\bigr]}
\end{equation*}
by Jensen's inequality. Also, using $k+1\leq2k$ for all $k\in\N$, we can bound 
\begin{align*}
 E\bigl[1_{\{N\geq1\}}N^{-1}\bigr]&=e^{-\lambda}\sum_{k=1}^\infty\frac{\lambda^k}{k k!}\leq 2e^{-\lambda}\sum_{k=1}^\infty\frac{\lambda^k}{(k+1) k!}
 =\frac{2}{\lambda}e^{-\lambda}\sum_{k=1}^\infty\frac{\lambda^{k+1}}{(k+1)!}\\
 &=\frac{2}{\lambda}e^{-\lambda}\sum_{l=2}^\infty\frac{\lambda^l}{l!}\leq\frac{2}{\lambda}\,.
\end{align*}
Hence, 
\begin{equation*}
 E\bigl[1_{\{N\geq1\}}N^{-1/2}\bigr]\leq\frac{\sqrt{2}}{\sqrt{\lambda}}\,.
\end{equation*}
Noting that 
\[\alpha=\gamma^2=\lambda\quad\text{and}\quad\sigma^2=\lambda(a^2+c^2)=\lambda b^2\,,\]
the result follows from Theorem \ref{maintheo}.\\
\end{proof}

\begin{remark}\label{poissrem}
 The Berry-Esseen bound presented in Corollary \ref{Poisson} is of the same order of $\lambda$ as the bound given in \cite{KorShev12}, which seems to be the best currently available, but has a worst constant. However, it should 
 be mentioned that the bound in \cite{KorShev12} was obtained using special properties of the Poisson distribution and does not seem likely to be easily transferable to other distributions of $N$.
\end{remark}

\begin{cor}\label{binomial}
Suppose that Assumption \ref{genass} holds, let $W$ be given by \eqref{defw} and let $Z$ have the standard normal distribution. Furthermore, assume that $N\sim\Bin(n,p)$ has the Binomial distribution with parameters $n\in\N$ and $p\in(0,1]$. Then, 
\begin{align*}
d_\W(W,Z)&\leq\frac{1}{\sqrt{np}\bigl(b^2-pa^2\bigr)^{3/2}}\biggl(\bigl(2c^2b+\abs{a}b^2\bigr)(1-p)+3 d^3\\
&\;+\sqrt{\frac{2}{\pi}}a^2p\sqrt{b^2-pa^2}\sqrt{1-p}\biggr)\quad\text{and}\\
d_\K(W,Z)&\leq\frac{1}{\sqrt{np}\bigl(b^2-pa^2\bigr)^{3/2}}\biggl(c^3+\frac{(\sqrt{2\pi}+4)bc^2\sqrt{1-p}}{4}+\frac{(3\sqrt{2\pi}+4) d^3}{8}\\
&\qquad+\frac{\abs{a}b^2\sqrt{1-p}}{2}+\frac{\abs{a}b^2\sqrt{2\pi}(1-p)}{8}\biggr)\\
&\;+\frac{1}{\sqrt{np}\bigl(b^2-pa^2\bigr)}\biggl(\Bigl(\frac{9}{2}\sqrt{2}+2\Bigr)\frac{ d^3}{c}+\sqrt{1-p}\bigl(a^2p+\sqrt{2}\abs{a}b d^3\bigr)\\
&\qquad+\frac{\sqrt{2(1-p)}b\bigl(2b^2-a^2\bigr)}{c\sqrt{2\pi}}\biggr)\\
&\;+\frac{c^2}{b^2-pa^2}(1-p)^n+\frac{\abs{a}b}{b^2-pa^2}(1-p)^{\frac{n+1}{2}}\,.
\end{align*}
\end{cor}

\begin{remark}\label{binrem}
Bounds for binomial random sums have also been derived in \cite{Sunk14} using a technique developed in \cite{Tik80}. Our bounds are of the same order $(np)^{-1/2}$ of magnitude.   
\end{remark}

\begin{proof}[Proof of Corollary \ref{binomial}]
Here, we clearly have 
\[\alpha=np\,,\quad \gamma^2=np(1-p)\quad\text{and}\quad\sigma^2=np(a^2(1-p)+c^2)\,.\]
Also, using the same coupling as in Example \ref{exdist} (b) we have $D\sim\Bern(1-p)$, 
\[E[D^2]=E[D]=1-p\quad\text{and}\quad E[D|N]=1-\frac{N}{n}\,.\]
This yields
\[\Var\bigl(E[D|N]\bigr)=\frac{1}{n^2}\Var(N)=\frac{p(1-p)}{n}\,.\]
We have $D^2=D$ and, by Cauchy-Schwarz, 
\begin{equation*}
 E\bigl[D1_{\{N\geq1\}}N^{-1/2}\bigr]\leq\sqrt{E[D^2]}\sqrt{E\bigl[1_{\{N\geq1\}}N^{-1}\bigr]}=\sqrt{1-p}\sqrt{E\bigl[1_{\{N\geq1\}}N^{-1}\bigr]}\,.
\end{equation*}
Using 
\begin{equation*}
 \frac{1}{k}\binom{n}{k}\leq\frac{2}{n+1}\binom{n+1}{k+1}\leq\frac{2}{n}\binom{n+1}{k+1}\,,\quad 1\leq k\leq n\,,
\end{equation*}
we have 
\begin{align*}
 E\bigl[1_{\{N\geq1\}}N^{-1}\bigr]&=\sum_{k=1}^n\frac{1}{k}\binom{n}{k}p^k(1-p)^{n-k}\leq\frac{2}{n}\sum_{k=1}^n\binom{n+1}{k+1}p^k (1-p)^{n-k}\\
 &=\frac{2}{np}\sum_{l=2}^{n+1}\binom{n+1}{l}p^l(1-p)^{n+1-l}\leq\frac{2}{np}\,.
 \end{align*}
Thus, 
\begin{equation*}
 E\bigl[D1_{\{N\geq1\}}N^{-1/2}\bigr]\leq\frac{\sqrt{2(1-p)}}{\sqrt{pn}}\quad\text{and}\quad E\bigl[1_{\{N\geq1\}}N^{-1/2}\bigr]\leq\frac{\sqrt{2}}{\sqrt{np}}\,.
\end{equation*}
Also, we can bound
\begin{equation*}
 E\Bigl[\bigl(E\bigl[D^2\,\bigl|\,N\bigr]\bigr)^2\Bigr]\leq E\bigl[D^4\bigr]=E[D]=1-p\,.
\end{equation*}
Now, using $a^2+c^2=b^2$, the claim follows from Theorem \ref{maintheo}.\\
\end{proof}


\begin{cor}\label{hyper}
 Suppose that Assumption \ref{genass} holds, let $W$ be given by \eqref{defw} and let $Z$ have the standard normal distribution. Assume further that $N\sim\Hyp(n;r,s)$ has the Hypergeometric distribution with parameters $n,r,s\in\N$ such that $n\leq\min\{r,s\}$. Then, 
 \begin{align*}
  d_\W(W,Z)&\leq    \Bigl(\frac{nr}{r+s}\Bigr)^{-1/2}\biggl(\frac{2b}{c}\;\frac{s(r+s-n)}{(r+s)_2}+\frac{3 d^3}{c^3}+\frac{\abs{a}b^2}{c^2}\;\frac{s(r+s-n)}{(r+s)_2}\biggr)\notag\\
  &\;+K\frac{a^2}{c^2}\sqrt{\frac{2}{\pi}}\biggl(\frac{\min\{r,s\}}{n(r+s)}\biggr)^{1/2}
  \quad\text{and}\\
  d_\K(W,Z)&\leq\Bigl(\frac{nr}{r+s}\Bigr)^{-1/2}\Biggl[1+\frac{(\sqrt{2\pi}+4)b}{4c}\Bigl(\frac{s(r+s-n)}{(r+s)_2}\Bigr)^{1/2}\\ 
  &\; +\Bigl(\frac{3\sqrt{2\pi}}{8}+\frac{9}{2}\sqrt{2}+\frac{5}{2}\Bigr)\frac{ d^3}{c^3}
  +\Bigl(\frac{\sqrt{2\pi}}{8}+1\Bigr)\frac{\abs{a}b^2}{c^3}\;\frac{s(r+s-n)}{(r+s)_2}\\
&\; + \Bigl(\abs{a}{b^2}{c^3\sqrt{2\pi}}+\frac{\abs{a}b d^3}{c^2}+\frac{b}{c\sqrt{2\pi}}\Bigr)\biggl(\frac{2s(r+s-n)}{(r+s)_2}\biggr)^{1/2}\Biggr]\\
&\;+\frac{(s)_n}{(r+s)_n}+K\frac{a^2}{c^2}\biggl(\frac{\min\{r,s\}}{n(r+s)}\biggr)^{1/2}+\frac{\abs{a}b}{c^2}\biggl(\frac{(s)_n}{(r+s)_n}\;\frac{s(r+s-n)}{(r+s)_2}\biggr)^{1/2}\,,
 \end{align*}
 where K is a numerical constant and $(m)_n=m(m-1)\cdot\ldots\cdot(m-n+1)$ denotes the lower factorial.
\end{cor}

\begin{proof}
In this case, we clearly have 
\begin{align*}
 \alpha&=\frac{nr}{r+s}\,,\quad\gamma^2=\frac{nr}{r+s}\;\frac{s}{r+s}\;\frac{r+s-n}{r+s-1}=\frac{nr}{r+s}\;\frac{s(r+s-n)}{(r+s)_2}\quad\text{and}\\
 \sigma^2&=\frac{nr}{r+s}\Bigl(c^2+a^2\frac{s}{r+s}\;\frac{r+s-n}{r+s-1}\Bigr)=\frac{nr}{r+s}\Bigl(c^2+a^2\frac{s(r+s-n)}{(r+s)_2}\Bigr)\,.
\end{align*}
Hence,
\begin{align*}
 c^2\frac{nr}{r+s}\leq\sigma^2\leq\frac{nr}{r+s}\Bigl(c^2+a^2\frac{s}{r+s}\Bigr)=\frac{nr}{r+s}\Bigl(b^2-a^2\frac{r}{r+s}\Bigr)\,.
\end{align*}

We use the coupling constructed in Example \ref{exdist} (c) but write $N$ for $X$ and $N^s$ for $X^s$, here. Recall that we have 
\begin{equation*}
 D=N^s-N=1_{\{X_1=0\}}\Bigl(1-\sum_{j=2}^n Y_j\Bigr)\geq0\quad\text{and}\quad D=D^2\,.
\end{equation*}
Furthermore, we know that 
\begin{equation*}
 E[D]=E[D^2]=d_\W(N,N^s)=\frac{\Var(N)}{E[N]}=\frac{s(r+s-n)}{(r+s)_2}\,.
\end{equation*}

Elementary combinatorics yield
\begin{equation*}
 E\bigl[Y_j\,\bigl|\,X_1,\dotsc,X_n\bigr]=r^{-1}1_{\{X_j=1\}}\,.
\end{equation*}
Thus, 
\begin{align*}
 E\bigl[D\,\bigl|\,X_1,\dotsc,X_n\bigr]&=1_{\{X_1=0\}}-\frac{1}{r}1_{\{X_1=0\}}\sum_{j=2}^n1_{\{X_j=1\}}=1_{\{X_1=0\}}\Bigl(1-\frac{N}{r}\Bigr)\quad\text{and}\\
 E\bigl[D\,\bigl|\,N\bigr]&=\Bigl(1-\frac{N}{r}\Bigr)P\bigl(X_1=0\,\bigl|\,N\bigr)=\Bigl(1-\frac{N}{r}\Bigr)\frac{n-N}{n}\\
 &=\frac{(r-N)(n-N)}{nr}=\Bigl(1-\frac{N}{r}\Bigr)\Bigl(1-\frac{N}{n}\Bigr)\,.
\end{align*}
Using a computer algebra system, one may check that
\begin{align}\label{eps}
 \Var\bigl(E\bigl[D\,\bigl|\,N\bigr]\bigr)&=\Bigl(n r s - n^3 r s - r^2 s + 5 n^2 r^2 s + 2 n^3 r^2 s - 8 n r^3 s - 
   8 n^2 r^3 s+   2 n r s^5 \notag   \\ 
   &\;  - n^3 r^3 s + 4 r^4 s + 10 n r^4 s + 3 n^2 r^4 s- 4 r^5 s - 3 n r^5 s + r^6 s + n s^2\notag\\ 
   &\;  - n^3 s^2 - 2 r s^2 +   4 n^2 r s^2 - 2 n^3 r s^2 - 14 n r^2 s^2 - 4 n^2 r^2 s^2+  n^3 r^2 s^2\notag\\ 
   &\; + 12 r^3 s^2 + 20 n r^3 s^2 + 2 n^2 r^3 s^2 - 
   14 r^4 s^2 - 7 n r^4 s^2 + 4 r^5 s^2 - s^3\notag\\
   &\; - n^2 s^3 + 2 n^3 s^3-   5 n r s^3 + 4 n^2 r s^3 + n^3 r s^3 + 13 r^2 s^3 + 8 n r^2 s^3\notag\\
   &\;-  4 n^2 r^2 s^3 - 18 r^3 s^3 - 3 n r^3 s^3 + 6 r^4 s^3 + n s^4- n^3 s^4 + 6 r s^4 - 4 n r s^4\notag\\ 
   &\;- 2 n^2 r s^4 - 10 r^2 s^4 + 
   3 n r^2 s^4 + 4 r^3 s^4 + s^5 - 2 n s^5 + n^2 s^5 - 2 r s^5\notag\\
   &\; +  r^2 s^5\Bigr)\Bigl(nr(r + s)^2(r+s-1)^2(r+s-2)(r+s-3)\Bigr)^{-1}\notag\\
  & =:\epsilon(n,r,s)\,.
\end{align}
One can check that under the assumption $n\leq\min\{r,s\}$ always 
\begin{equation*}
 \epsilon(n,r,s)=O\biggl(\frac{\min\{r,s\}}{n(r+s)}\biggr)\,.
\end{equation*}
Hence, there is a numerical constant $K$ such that 
\begin{equation*}
 \sqrt{\epsilon(n,r,s)}\leq K\biggl(\frac{\min\{r,s\}}{n(r+s)}\biggr)^{1/2}\,.
\end{equation*}

Also, by the conditional version of Jensen's inequality
\begin{equation*}
 E\Bigl[\bigl(E\bigl[D^2\,\bigl|\,N\bigr]\bigr)^2\Bigr]\leq E\bigl[D^4\bigr]=E[D]=\frac{s(r+s-n)}{(r+s)_2}\,.
\end{equation*}
Using 
\begin{align*}
 E\bigl[N^{-1}1_{\{N\geq1\}}\bigr]&=\binom{r+s}{n}^{-1}\sum_{k=1}^n\frac{1}{k}\binom{r}{k}\binom{s}{n-k}\\
&\leq \frac{2}{r+1}\binom{r+s}{n}^{-1}\sum_{l=2}^{n+1}\binom{r+1}{l}\binom{s}{n+1-l}\\
&\leq \frac{2}{r+1}\binom{r+s}{n}^{-1}\binom{r+1+s}{n+1}=\frac{2(r+s+1)}{(n+1)(r+1)}\leq2\frac{r+s}{nr}\,,
\end{align*}
we get
\begin{align*}
E\bigl[D1_{\{N\geq1\}}N^{-1/2}\bigr]&\leq\sqrt{E[D^2]}\sqrt{E\bigl[1_{\{N\geq1\}}N^{-1}\bigr]}\leq\biggl(2\frac{s(r+s-n)}{(r+s)(r+s-1)}\;\frac{r+s}{nr}\biggr)^{1/2}\\
&=\sqrt{2}\Bigl(\frac{s}{nr}\;\frac{r+s-n}{r+s-1}\Bigr)^{1/2}\quad\text{and}\\
E\bigl[1_{\{N\geq1\}}N^{-1/2}\bigr]&\leq\sqrt{E\bigl[1_{\{N\geq1\}}N^{-1}\bigr]}\leq\sqrt{2}\Bigl(\frac{r+s}{nr}\Bigr)^{1/2}\,.
\end{align*}
Finally, we have
\begin{equation*}
 P(N=0)=\frac{\binom{s}{n}}{\binom{r+s}{n}}=\frac{s(s-1)\cdot\ldots\cdot(s-n+1)}{(r+s)(r+s-1)\cdot\ldots\cdot(r+s-n+1)}=\frac{(s)_n}{(r+s)_n}\,.
\end{equation*}
Thus, the result follows from Theorem \ref{maintheo}.\\
\end{proof}

\begin{remark}\label{hyprem}
\begin{enumerate}[(a)]
\item From the above proof we see that the numerical constant $K$ appearing in the bounds of Corollary \ref{hyper} could in principle be computed explicitly. 
Also, as always 
\begin{equation*}
 \frac{\min\{r,s\}}{n(r+s)}\leq\frac{r+s}{nr}=\frac{1}{E[N]}\,,
\end{equation*}
we conclude that the bounds are of order $E[N]^{-1/2}$. 
\item One typical situation, in which a CLT for Hypergeometric random sums holds, is when $N$, itself, is asymptotically normal. 
Using the same coupling $(N,N^s)$ as in the above proof and the results from \cite{GolRin96}, one obtains that under the condition 
\begin{equation}\label{asno}
 \frac{\max\{r,s\}}{n\min\{r,s\}}\longrightarrow0
\end{equation}
the index $N$ is asymptotically normal. This condition is stricter than that 
\begin{equation}\label{asnors}
 E[N]^{-1}=\frac{r+s}{nr}\longrightarrow0\,,
\end{equation}
which implies the random sums CLT. For instance, choosing
\begin{equation*}
r\propto n^{1+\eps}\,,\quad\text{and}\quad s\propto n^{1+\kappa}
\end{equation*}
with $\eps,\kappa\geq0$, then \eqref{asno} holds, if and only if $\abs{\eps-\kappa}<1$, whereas \eqref{asnors} is equivalent to $\kappa-\eps<1$ in this case.
\end{enumerate}
\end{remark}

Before we end this section by giving the proof of Lemma \ref{distlemma}, we would like to mention in what respects the results in this article could be generalized. Firstly, it would be possible do 
dispense with the assumption of independence among the summands $X_1,X_2,\dotsc$. Of course, the terms appearing in the bounds would look more complicated, but the only essential change would be the emergence of the additional error term
\begin{equation*}
 E_3:=\frac{C}{\alpha\sigma} E\babs{\alpha A_N-\mu N}\,,
\end{equation*}
where 
\begin{equation*}
 A_N:=\sum_{j=1}^N E[X_j]\quad\text{and}\quad \mu=E[A_N]
\end{equation*}
and where $C$ is an explicit constant depending on the probabilistic distance chosen. 
Note that $E_3=0$ if the summands are either i.i.d. or centered. 

Secondly, it would be possible in principle to allow for some dependence among the summands $X_1,X_2,\dotsc$. Indeed, an inspection of the proof in Section \ref{proof} reveals that 
this dependence should be such that for the non-random partial sums bounds on the normal approximation exist and such that suitable couplings with the non-zero biased distribution (see Section \ref{stein} ) 
of those partial sums are available. 
The latter, however, have not been constructed yet in great generality, although \cite{GolRei97} gives a construction for summands forming a simple random sampling in the zero bias case.

It would be much more difficult to abandon the assumption about the independence of the summation index and the summands. This can be seen from Equation \eqref{mp8} below, in which the second identity 
would no longer hold, in general, if this independence was no longer valid. Also, one would no longer be able to freely choose the coupling $(N,N^s)$ when specializing to concrete distributions of $N$. 

\begin{proof}[Proof of Lemma \ref{distlemma}]
Let $h$ be a measurable function such that all the expected values in \eqref{sbdef} exist. By \eqref{sbdef} we have 
\begin{align}\label{sb1}
\Bigl|E[h(X^s)]-E[h(X)]\Bigr|&=\frac{1}{E[X]}\Bigr|E\bigl[\bigl(X-E[X]\bigr)h(X)\bigr]\Bigr|\,.
\end{align}
It is well known that 
\begin{equation}\label{sb2}
d_{TV}(X,Y)=\sup_{h\in\calH}\Bigl|E[h(X)]-E[h(Y)]\Bigr|\,,
\end{equation}
where $\calH$ is the class of all measurable functions on $\R$ such that $\fnorm{h}\leq1/2$.
If $\fnorm{h}\leq1/2$, then 
\begin{equation*}
\frac{1}{E[X]}\Bigr|E\bigl[\bigl(X-E[X]\bigr)h(X)\bigr]\Bigr|\leq\frac{E\abs{X-E[X]}}{2E[X]}\,.
\end{equation*}
Hence, from \eqref{sb2} and \eqref{sb1} we conclude that 
\begin{equation*}\label{sb3}
d_{TV}(X,X^s)\leq\frac{E\abs{X-E[X]}}{2E[X]}\,.
\end{equation*}
On the other hand, letting 
\[h(x):=\frac{1}{2}\Bigl(1_{\{x>E[X]\}}-1_{\{x\leq E[X]\}}\Bigr)\]
in \eqref{sb1} we have $h\in\calH$ and obtain 
\begin{equation*}\label{sb4}
\Bigl|E[h(X^s)]-E[h(X)]\Bigr|=\frac{E\abs{X-E[X]}}{2E[X]}
\end{equation*}
proving the second equality of (a). 
Note that, since $X^s$ is stochastically larger than $X$, we have 
\begin{align}\label{sb8}
d_\K(X,X^s)&=\sup_{t\geq0}\bigl|P(X^s>t)-P(X>t)\bigr|=\sup_{t\geq0}\Bigl(P(X^s>t)-P(X>t)\Bigr)\nonumber\\
&=\sup_{t\geq0}\Bigl(E[g_t(X^s)]-E[g_t(X)]\Bigr)\,,
\end{align}
where $g_t:=1_{(t,\infty)}$.\\
By \eqref{sb1}, choosing $t=E[X]$ yields
\begin{equation}\label{sb10}
 d_\K(X,X^s)\geq E\bigl[\bigl(X-E[X]\bigr)1_{\{X>E[X]\}}\bigr]\,.
\end{equation}
If $0\leq t<E[X]$ we obtain 
\begin{align}\label{sb5}
E\bigl[\bigl(X-E[X]\bigr)1_{\{X>t\}}\bigr]&=E\bigl[\bigl(X-E[X]\bigr)1_{\{t<X\leq E[X]\}}\bigr]
+E\bigl[\bigl(X-E[X]\bigr)1_{\{X>E[X]\}}\bigr]\nonumber\\
&\leq E\bigl[\bigl(X-E[X]\bigr)1_{\{X>E[X]\}}\bigr]\,.
\end{align}
Also, if $t\geq E[X]$, then 
\begin{equation}\label{sb6}
E\bigl[\bigl(X-E[X]\bigr)1_{\{X>t\}}\bigr]\leq E\bigl[\bigl(X-E[X]\bigr)1_{\{X>E[X]\}}\bigr]\,.
\end{equation}
Thus, by \eqref{sb1}, from \eqref{sb8}, \eqref{sb10}, \eqref{sb5} and \eqref{sb6} we conclude that 
\begin{equation}\label{sb9}
 d_\K(X,X^s)= E\bigl[\bigl(X-E[X]\bigr)1_{\{X>E[X]\}}\bigr]\,.
\end{equation}
Now, the remaining claim of (a) can be easily inferred from \eqref{sb9} and from the following two identities:
\begin{align*}
 0&=E\bigl[X-E[X]\bigr]=E\bigl[\bigl(X-E[X]\bigr)1_{\{X>E[X]\}}\bigr]-E\bigl[\bigl(X-E[X]\bigr)1_{\{X\leq E[X]\}}\bigr]\\
 &=E\bigl[\babs{X-E[X]}1_{\{X>E[X]\}}\bigr]-E\bigl[\babs{X-E[X]}1_{\{X\leq E[X]\}}\bigr]
 \end{align*}
 and
 \begin{align*}
 E\babs{X-E[X]}&=E\bigl[\babs{X-E[X]}1_{\{X>E[X]\}}\bigr]+E\bigl[\babs{X-E[X]}1_{\{X\leq E[X]\}}\bigr]\\
 &=2E\bigl[\bigl(X-E[X]\bigr)1_{\{X>E[X]\}}\bigr]\,.
\end{align*}
Finally, if $h$ is $1$-Lipschitz continuous, then 
\begin{align*}
\Bigr|E\bigl[\bigl(X-E[X]\bigr)h(X)\bigr]\Bigr|&=\Bigr|E\bigl[\bigl(X-E[X]\bigr)\bigl(h(X)-h(E[X])\bigr)\bigr]\Bigr|\\
&\leq \fnorm{h'}E\bigl[\abs{X-E[X]}^2\bigr]=\Var(X)\,.
\end{align*} 
On the other hand, the function $h(x):=x-E[X]$ is $1$-Lipschitz and 
\begin{equation*}
E\bigl[\bigl(X-E[X]\bigr)h(X)\bigr]=\Var(X)\,.
\end{equation*}
Thus, also (b) is proved.\\
\end{proof}

\section{Elements of Stein's method}\label{stein}
In this section we review some well-known and also some recent results about Stein's method of normal approximation. Our general reference for this topic is the book \cite{CGS}. Throughout, $Z$ will denote a standard normal random variable. Stein's method originated from Stein's seminal observation 
(see \cite{St72}) that a real-valued random variable $X$ has the standard normal distribution, if and only if the identity
\begin{equation*}\label{steinchar}
 E\bigl[f'(X)\bigr]=E\bigl[Xf(X)\bigr]
\end{equation*}
holds for each, say, continuously differentiable function $f$ with bounded derivative. For a given random variable $W$, which is supposed to be asymptotically normal, and a Borel-measurable test function $h$ on $\R$ with $E\abs{h(Z)}<\infty$ it was then Stein's idea to solve the \textit{Stein equation}  
\begin{equation}\label{steineq}
f'(x)-xf(x)=h(x)-E[h(Z)]
\end{equation}
and to use properties of the solution $f$ and of $W$ in order to bound the right hand side of
\begin{equation*}
\Babs{E\bigl[h(W)\bigr]-E\bigl[h(Z)\bigr]}=\Babs{E\bigl[f'(W)-Wf(W)\bigr]} 
\end{equation*}
rather than bounding the left hand side directly. For $h$ as above, by $f_h$ we denote the standard solution to the Stein equation \eqref{steineq}
which is given by 
\begin{align}\label{standsol}
f_h(x)&=e^{x^2/2}\int_{-\infty}^x\bigl(h(t)-E[h(Z)]\bigr)e^{-t^2/2}dt\nonumber\\
&=-e^{x^2/2}\int_{x}^\infty\bigl(h(t)-E[h(Z)]\bigr)e^{-t^2/2}dt\,.
\end{align}
Note that, generally, $f_h$ is only differentiable and satisfies \eqref{steineq} at the continuity points of $h$. 
In order to be able to deal with distributions which might have point masses, if $x\in\R$ is a point at which $f_h$ is not differentiable, one defines
\begin{equation}\label{deffprime}
f_h'(x):=xf_h(x)+h(x)-E[h(Z)]
\end{equation} 
such that, by definition, $f_h$ satisfies \eqref{steineq} at each point $x\in\R$. This gives a Borel-measurable version of 
the derivative of $f_h$ in the Lebesgue sense.
Properties of the solutions $f_h$ for various classes of test functions $h$ have been studied. Since we are only interested in the Kolmogorov and Wasserstein distances, we either suppose that $h$ is $1$-Lipschitz or that 
$h=h_z=1_{(-\infty,z]}$ for some $z\in\R$. In the latter case we write $f_z$ for $f_{h_z}$.\\
We need the following properties of the solutions $f_h$. If $h$ is $1$-Lipschitz, then it is well known (see e.g. \cite{CGS}) that 
$f_h$ is continuously differentiable and that both $f_h$ and $f_h'$ are Lipschitz-continuous with 
\begin{equation}\label{bounds}
\fnorm{f_h}\leq1\,,\quad\fnorm{f_h'}\leq\sqrt{\frac{2}{\pi}}\quad\text{and }\fnorm{f_h''}\leq2\,.
\end{equation} 
Here, for a function $g$ on $\R$, we denote by 
\[\fnorm{g'}:=\sup_{x\not=y}\frac{\abs{g(x)-g(y)}}{\abs{x-y}}\]
its minimum Lipschitz constant. Note that if $g$ is absolutely continuous, then $\fnorm{g_h'}$ coincides with the essential supremum norm of the derivative of $g$ in the Lebesgue sense. Hence, the double use of the symbol $\fnorm{\cdot}$ does not 
cause any problems. For an absolutely continuous function $g$ on $\R$, a fixed choice of its derivative $g'$ and for $x,y\in\R$ we let 
\begin{equation}\label{resttaylor}
R_g(x,y):=g(x+y)-g(x)-g'(x)y
\end{equation}
denote the remainder term of its first order Taylor expansion around $x$ at the point $x+y$.
If $h$ is $1$-Lipschitz, then we obtain for all $x,y\in\R$ that 
\begin{equation}\label{fhtaylor}
\babs{R_{f_h}(x,y)}=\babs{f_h(x+y)-f_h(x)-f_h'(x)y}\leq y^2\,.
\end{equation}
This follows from \eqref{bounds} via
\begin{align*}
&\babs{f_h(x+y)-f_h(x)-f_h'(x)y}=\Babs{\int_x^{x+y}\bigl(f_h'(t)-f_h'(x)\bigr)dt}\\
&\;\leq\fnorm{f_h''}\Babs{\int_x^{x+y}\abs{t-x}dt}=\frac{y^2\fnorm{f_h''}}{2}\leq y^2\,.
\end{align*}
For $h=h_z$ we list the following properties of $f_z$:
The function $f_z$ has the representation
\begin{equation}\label{steinsol}
f_z(x)=\begin{cases}
\frac{(1-\Phi(z))\Phi(x)}{\phi(x)}\,,&x\leq z\\
\frac{\Phi(z)(1-\Phi(x))}{\phi(x)}\,,&x>z\,.
\end{cases}
\end{equation}
Here, $\Phi$ denotes the standard normal distribution function and $\phi:=\Phi'$ the corresponding continuous density. 
It is easy to see from \eqref{steinsol} that $f_z$ is infinitely often differentiable on $\R\setminus\{z\}$. 
Furthermore, it is well-known that $f_z$ is Lipschitz-continuous with Lipschitz constant $1$ and that it satisfies 
\begin{equation*}\label{boundfz}
0<f_z(x)\leq f_0(0)=\frac{\sqrt{2\pi}}{4}\,,\quad x,z\in\R\,.
\end{equation*}
These properties already easily yield that for all $x,u,v,z\in\R$
\begin{equation}\label{fzdiff}
\babs{(x+u)f_z(x+u)-(x+v)f_z(x+v)}\leq\Biggl(\abs{x}+\frac{\sqrt{2\pi}}{4}\Biggr)\bigl(\abs{u}+\abs{v}\bigr)\,.
\end{equation}
Proofs of the above mentioned classic facts about the functions $f_z$ can again be found in \cite{CGS}, for instance.
As $f_z$ is not differentiable at $z$ (the right and left derivatives do exist but are not equal) by the above Convention \eqref{deffprime} we define 
\begin{equation}\label{deffzprime}
f_z'(z):=zf_z(z)+1-\Phi(z)
\end{equation}
such that $f=f_z$ satisfies \eqref{steineq} with $h=h_z$ for all $x\in\R$. Furthermore, with this definition, 
for all $x,z\in\R$ we have 
\begin{equation}\label{boundfzp}
\abs{f_z'(x)}\leq1\,.
\end{equation}
The following quantitative version of the first order Taylor approximation of $f_z$ has recently been proved by 
Lachi\`{e}ze-Rey and Peccati \cite{LaPec15} and had already been used implicitly in \cite{EiThae14}. Using \eqref{deffzprime}, for all $x,u,z\in\R$ we have 
\begin{align}\label{fztaylor}
\babs{R_{f_z}(x,u)}&=\babs{f_z(x+u)-f_z(x)-f_z'(x)u}\notag\\
&\leq\frac{u^2}{2}\Biggl(\abs{x}+\frac{\sqrt{2\pi}}{4}\Biggr)
+\abs{u}\Bigl(1_{\{x<z\leq x+u\}}+1_{\{x+u\leq z<x\}}\Bigr)\notag\\
&=\frac{u^2}{2}\Biggl(\abs{x}+\frac{\sqrt{2\pi}}{4}\Biggr)
+\abs{u}1_{\bigl\{z-(u\vee 0)<x\leq z-(u\wedge 0)\bigr\}}\,,
\end{align}
where, here and elsewhere, we write $x\vee y:=\max(x,y)$ and $x\wedge y:=\min(x,y)$.\\

For the proof of Theorems \ref{maintheo} and \ref{meanzero} we need to recall a certain coupling construction, which has been efficiently used in Stein's method of normal approximation: 
Let $X$ be a real-valued random variable such that $E[X]=0$ and $0<E[X^2]<\infty$. In \cite{GolRei97} it was proved that there exists a unique distribution 
for a random variable $X^*$ such that for all Lipschitz continuous functions $f$ the identity 
\begin{equation}\label{zbdef}
E[Xf(X)]=\Var(X)E[f'(X^*)]
\end{equation}
holds true. The distribution of $X^*$ is called the \textit{$X$-zero biased distribution} and the distributional transformation which maps $\calL(X)$ to 
$\calL(X^*)$ is called the \textit{zero bias transformation}. It can be shown that \eqref{zbdef} holds for all absolutely continuous functions $f$ on $\R$ such that $E\abs{Xf(X)}<\infty$. From the Stein characterization of the family of normal distributions it is immediate that the fixed points of the zero bias transformation are exactly the centered normal distributions. Thus, if, for a given $X$, the distribution of $X^*$ is close to that of $X$, the distribution of $X$ is approximately a fixed point of this transformation and, hence, should be close to the normal distribution with the same variance as $X$. In \cite{Gol04} this heuristic was made precise by showing the inequality 
\begin{equation*}\label{wasineq}
d_\W(X,\sigma Z)\leq2d_\W(X,X^*)\,,
\end{equation*}   
where $X$ is a mean zero random variable with $0<\sigma^2=E[X^2]=\Var(X)<\infty$, $X^*$ having the $X$-zero biased distribution is defined on the same probability space as $X$ and $Z$ is standard normally distributed. 
For merely technical reasons we introduce a variant of the zero bias transformation for not necessarily centered random variables. Thus, if $X$ is a real random variable with $0<E[X^2]<\infty$, we say that a random variable $X^{nz}$ has the \textit{$X$-non-zero biased distribution}, if for all Lipschitz-continuous functions $f$ it holds that
\begin{equation*}\label{nonzero}
E\bigl[\bigl(X-E[X]\bigr)f(X)\bigr]=\Var(X)E\bigl[f'(X^{nz})\bigr]\,.
\end{equation*}
Existence and uniqueness of the $X$-non-zero biased distribution immediately follow from Theorem 2.1 of \cite{GolRei05b} (or Theorem 2.1 of \cite{Doe13b} by letting 
$B(x)=x-E[X]$, there). Alternatively, letting $Y:=X-E[X]$ and $Y^{*}$ have the $Y$-zero biased distribution, it is easy to see that $X^{nz}:=Y^{*}+E[X]$ fulfills the requirements for the $X$-non-zero biased distribution. Most of the properties of the zero bias transformation have natural analogs for the non-zero bias transformation, so we do not list them all, here. Since an important part of the proof of our main result relies on the so-called \textit{single summand property}, 
however, we state the result for the sake of reference. 

\begin{lemma}[single summand property]\label{single}
Let $X_1,\dots,X_n$ be independent random variables such that $0<E[X_j^2]<\infty$, $j=1,\dotsc,n$. Define $\sigma_j^2:=\Var(X_j)$,\\ $j=1,\dotsc,n$, $S:=\sum_{j=1}^nX_j$ 
and $\sigma^2:=\Var(S)=\sum_{j=1}^n \sigma_j^2$. 
For each $j=1,\dotsc,n$ let $X_j^{nz}$ have the $X_j$-non-zero biased distribution and be independent of\\ $X_1,\dotsc,X_{j-1},X_{j+1},\dotsc,X_n$ and let 
$I\in\{1,\dotsc,n\}$ be a random index, independent of all the rest and such that 
\[P(I=j)=\frac{\sigma_j^2}{\sigma^2}\,,\quad j=1,\dotsc,n\,.\]
Then, the random variable 
\[S^{nz}:=S-X_I+X_I^{nz}=\sum_{i=1}^n 1_{\{I=i\}}\Bigl(\sum_{j\not=i}X_j +X_i^{nz}\Bigr)\]
has the $S$-non-zero biased distribution.
\end{lemma}

\begin{proof}
The proof is either analogous to the proof of Lemma 2.1 in \cite{GolRei97} or else, the statement could be deduced from this result in the following way:
Using the fact that $X^{nz}=Y^{*}+E[X]$ has the $X$-non-zero biased distribution if and only if $Y^*$ has the $(X-E[X])$-zero biased distribution, we    
Let $Y_j:=X_j-E[X_j]$, $Y_j^{*}:=X_j^{nz}-E[X_j]$, $j=1,\dotsc,n$ and $W:=\sum_{j=1}^n Y_j=S-E[S]$. Then, from Lemma 2.1 in \cite{GolRei97} we know that 
\begin{align*}
W^*&:=W-Y_I+Y_I^{*}=S-E[S]+\sum_{j=1}^n1_{\{I=j\}}\Bigl(E[X_j]-X_j+X_j^{nz}-E[X_j]\Bigr)\\
&=S-X_I+X_I^{nz}-E[S]=S^{nz}-E[S]
\end{align*}
has the $W$-zero biased distribution, implying that $S^{nz}$ has the $S$-non-zero biased distribution.\\
\end{proof} 

\section{Proof of Theorems \ref{maintheo} and \ref{meanzero}}\label{proof}
From now on we let $h$ be either $1$-Lipschitz or $h=h_z$ for some $z\in\R$ and write $f=f_h$ given by \eqref{standsol}.
Since $f$ is a solution to \eqref{steineq}, plugging in $W$ and taking expectations yields 
\begin{equation}\label{mp1}
E[h(W)]-E[h(Z)]=E[f'(W)-Wf(W)]\,.
\end{equation}
As usual in Stein's method of normal approximation, the main task is to rewrite the term $E[Wf(W)]$ into a more tractable expression be exploiting the structure 
of $W$ and using properties of $f$. From \eqref{defw} we have 
\begin{equation}\label{mp2}
E[Wf(W)]=\frac{1}{\sigma}E[(S-aN)f(W)]+\frac{a}{\sigma}E[(N-\alpha)f(W)]=:T_1+T_2\,.
\end{equation}
For ease of notation, for $n\in\Z_+$ and $M$ any $\Z_+$-valued random variable we let
\begin{equation*}\label{notation}
S_n:=\sum_{j=1}^n X_j,\quad W_n:=\frac{S_n-\alpha a}{\sigma},\quad S_M:=\sum_{j=1}^M X_j\quad\text{and }W_M:=\frac{S_M-\alpha a}{\sigma}\,,
\end{equation*} 
such that, in particular, $S=S_N$ and $W=W_N$. 
Using the decomposition
\[E\bigl[f'(W)\bigr]=\frac{\alpha c^2}{\sigma^2}E\bigl[f'(W)\bigr]+\frac{a^2\gamma^2}{\sigma^2}E\bigl[f'(W)\bigr]\]
which is true by virtue of \eqref{meanvar}, from \eqref{mp1} and \eqref{mp2} we have 
\begin{align}\label{esdec}
E[h(W)]-E[h(Z)]&=E[f'(W)]-T_1-T_2\notag\\
&=E\Bigl[\frac{c^2\alpha}{\sigma^2}f'(W)-\frac{1}{\sigma}(S-aN)f(W)\Bigr]\notag\\
&\;+E\Bigl[\frac{a^2\gamma^2}{\sigma^2}f'(W)-\frac{a}{\sigma}(N-\alpha)f(W)\Bigr]=:E_1+E_2\,.
\end{align}
We will bound the terms $E_1$ and $E_2$ seperately. 
Using the independence of $N$ and $X_1,X_2,\dotsc$ for $T_1$ we obtain:
\begin{align}\label{mp3}
T_1&=\frac{1}{\sigma}\sum_{n=0}^\infty P(N=n)E\bigl[(S_n-na)f(W_n)\bigr]\notag\\
&=\frac{1}{\sigma}\sum_{n=0}^\infty P(N=n)E\bigl[(S_n-na)g(S_n)\bigr]\,,
\end{align}
where 
\[g(x):=f\left(\frac{x-\alpha a}{\sigma}\right)\,.\] 
Thus, if, for each $n\geq0$, $S_n^{nz}$ has the $S_n$-non-zero biased distribution, from \eqref{mp3} and \eqref{nonzero} we obtain that 
\begin{align*}\label{mp4}
T_1&=\frac{1}{\sigma}\sum_{n=0}^\infty P(N=n)\Var(S_n)E\bigl[g'\bigl(S_n^{nz}\bigr)\bigr]\notag\\
&=\frac{c^2}{\sigma^2}\sum_{n=0}^\infty nP(N=n)E\Bigl[f'\Bigl(\frac{S_n^{nz}-\alpha a}{\sigma}\Bigr)\Bigr]\,.
\end{align*}
Note that if we let $M$ be independent of $S_1^{nz}, S_2^{nz},\dotsc$ and have the $N$-size biased distribution, then,  this implies that 
\begin{equation}\label{mp5}
T_1=\frac{c^2\alpha}{\sigma^2}E\Bigl[f'\Bigl(\frac{S_{M}^{nz}-\alpha a}{\sigma}\Bigr)\Bigr]\,,
\end{equation}
 where
\[S_{M}^{nz}=\sum_{n=1}^{\infty}1_{\{M=n\}}S_n^{nz}\,.\]
We use Lemma \ref{single} for the construction of the variables $S_n^{nz}$, $n\in\N$.
Note, however, that by the i.i.d. property of the $X_j$ we actually do not need the mixing index $I$, here.
Hence, we construct independent random variables 
\[(N,M), X_1,X_2,\dotsc\text{ and }Y\]
such that $M$ has the $N$-size biased distribution and such that $Y$ has the $X_1$-non-zero biased distribution. 
Then, for all $n\in\N$ 
\[S_n^{nz}:=S_n-X_1+Y\]
has the $S_n$-non-zero biased distribution and we have 
\begin{equation}\label{mp6}
\frac{S_{M}^{nz}-\alpha a}{\sigma}=\frac{S_{M}-\alpha a}{\sigma}+\frac{Y-X_1}{\sigma}=W_{M}+\frac{Y-X_1}{\sigma}=:W^*\,.
\end{equation}
Thus, from \eqref{mp6} and \eqref{mp5} we conclude that 
\begin{equation}\label{mp7}
T_1=\frac{c^2\alpha}{\sigma^2}E\bigl[f'(W^*)]
\end{equation}
and 
\begin{equation}\label{e1}
 E_1=\frac{c^2\alpha}{\sigma^2}E\bigl[f'(W)-f'(W^*)\bigr]\,.
\end{equation}
We would like to mention that if $a=0$, then, by \eqref{mp7}, $W^*$ has the $W$-zero biased distribution as $T_2=0$ and $\sigma^2=c^2\alpha$ in this case.
Before addressing $T_2$, we remark that the random variables appearing in $E_1$ and $E_2$, respectively, could possibly be defined on different probability spaces, if convenient,
since they do not appear under the same expectation sign. Indeed, for $E_2$ we use the coupling 
$(N,N^s)$, which is given in the statements of Theorems \ref{maintheo} and \ref{meanzero} and which appears in the bounds via the difference $D=N^s-N$. In order to manipulate $E_2$ we thus assume that the random variables 
\begin{equation*}
 (N,N^s),X_1,X_2,\dotsc
\end{equation*}
are independent and that $N^s$ has the $N$-size biased distribution. Note that we do not assume here that $D=N^s-N\geq0$, since sometimes a natural coupling yielding a small value of $\abs{D}$ does not satisfy this nonnegativity condition.
In what follows we will use the notation 
\begin{equation*}\label{defv}
V:=W_{N^s}-W_N=\frac{1}{\sigma}\bigl(S_{N^s}-S_N\bigr)\quad\text{and}\quad J:=1_{\{D\geq0\}}=1_{\{N^s\geq N\}}\,.
\end{equation*}
Now we turn to rewriting $T_2$. Using the independence of $N$ and $X_1,X_2,\dotsc$, and that of $N^s$ and $X_1,X_2,\dotsc$, respectively, $E[N]=\alpha$ 
and the defining equation \eqref{sbdef} of the $N$-size biased distribution, we obtain from \eqref{mp2} that
\begin{align}\label{mp8}
T_2&=\frac{a}{\sigma}E[(N-\alpha)f(W_N)]=\frac{\alpha a}{\sigma}E\bigl[f(W_{N^s})-f(W_N)\bigr]\notag\\
&=\frac{\alpha a}{\sigma}E\Bigl[1_{\{N^s\geq N\}}\bigl(f(W_{N^s})-f(W_N)\bigr)\Bigr]
+\frac{\alpha a}{\sigma}E\Bigl[1_{\{N^s< N\}}\bigl(f(W_{N^s})-f(W_N)\bigr)\Bigr]\notag\\
&=\frac{\alpha a}{\sigma}E\bigl[J\bigl(f(W_N+V)-f(W_N)\bigr)\bigr]
-\frac{\alpha a}{\sigma}E\bigl[(1-J)\bigl(f(W_{N^s}-V)-f(W_{N^s})\bigr)\bigr]\notag\\
&=\frac{\alpha a}{\sigma}E\bigl[JVf'(W_N)\bigr]
+\frac{\alpha a}{\sigma}E\bigl[JR_f(W_N,V)\bigr]\notag\\
&\;+\frac{\alpha a}{\sigma}E\bigl[(1-J)Vf'(W_{N^s})\bigr]
-\frac{\alpha a}{\sigma}E\bigl[(1-J)R_f(W_{N^s},-V)\bigr]\,,
\end{align}
where $R_f$ was defined in \eqref{resttaylor}. Note that we have 
\[JV=1_{\{N^s\geq N\}}\frac{1}{\sigma}\sum_{j=N+1}^{N^s}X_j\quad\text{and}\quad  W_N=\frac{\sum_{j=1}^N X_j-\alpha a}{\sigma}\]
and, hence, the random variables $JV$ and $W_N$
are conditionally independent given $N$. Noting also that 
\begin{align*}
E\bigl[JV\,\bigl|\,N\bigr]&=\frac{1}{\sigma}E\biggl[J\sum_{j=N+1}^{N^s}X_j\,\biggl|\,N\biggr]\\
&=\frac{1}{\sigma}E\biggl[JE\Bigl[\sum_{j=N+1}^{N^s}X_j\,\Bigl|\,N,N^s\Bigr]\,\biggl|\,N\biggr]\\
&=\frac{a}{\sigma}E\bigl[JD\,\bigl|\,N\bigr]=\frac{a}{\sigma}E\bigl[JD\,\bigl|\,N\bigr]
\end{align*}
we obtain that 
\begin{align}\label{mpn1}
\frac{\alpha a}{\sigma}E\bigl[JVf'(W_N)\bigr]
&=\frac{\alpha a}{\sigma}E\bigl[E\bigl[JV\,\bigl|\,N\bigr]E\bigl[f'(W_N)\,\bigl|\,N\bigr]\Bigr]\notag\\
&=\frac{\alpha a^2}{\sigma^2}E\Bigl[E\bigl[JD\,\bigl|\,N\bigr]E\bigl[f'(W_N)\,\bigl|\,N\bigr]\Bigr]\notag\\
&=\frac{\alpha a^2}{\sigma^2}E\Bigl[E\bigl[JDf'(W_N)\,\bigl|\,N\bigr]\Bigr]\notag\\
&=\frac{\alpha a^2}{\sigma^2}E\Bigl[JDf'(W_N)\Bigr]\,,
\end{align}
where we have used for the next to last equality that also $D$ and $W_N$
are conditionally independent given $N$. In a similar fashion, using that $W_{N^s}$ and $1_{\{D<0\}}V$  and also $W_{N^s}$ and $D$ are conditionally independent given $N^s$, one can show
\begin{align}\label{mpn2}
\frac{\alpha a}{\sigma}E\bigl[(1-J)Vf'(W_{N^s})\bigr]
=\frac{\alpha a^2}{\sigma^2}E\bigl[(1-J)Df'(W_{N^{s}})\bigr]\,.
\end{align}
Hence, using that 
\[\frac{\alpha a^2}{\sigma^2}E[D]=\frac{\alpha a^2}{\sigma^2}E[N^s-N] =\frac{\alpha a^2}{\sigma^2}\frac{\gamma^2}{\alpha}=\frac{a^2\gamma^2}{\sigma^2}\]
from \eqref{esdec}, \eqref{mp8}, \eqref{mpn1} and \eqref{mpn2}  we obtain
\begin{align}\label{e2}
 E_2&=\frac{\alpha a^2}{\sigma^2}E\Bigl[\bigl(E[D]-D\bigr)f'(W_N)\Bigr]+\frac{\alpha a^2}{\sigma^2}E\Bigl[(1-J)D\bigl(f'(W_N)-f'(W_{N^s})\bigr)\Bigr]\notag\\
 &\;-\frac{\alpha a}{\sigma}E\bigl[JR_f(W_N,V)\bigr]+\frac{\alpha a}{\sigma}E\bigl[(1-J)R_f(W_{N^s},-V)\bigr]\notag\\
 &=:E_{2,1}+E_{2,2}+E_{2,3}+E_{2,4}\,.
\end{align}
Using the conditional independence of $D$ and $W_N$ given $N$ as well as the Cauchy-Schwarz inequality, we can estimate
\begin{align}\label{e21}
\abs{E_{2,1}}&=\frac{\alpha a^2}{\sigma^2}\Babs{E\Bigl[E\bigl[D-E[D]\,\bigr|\,N\bigr]E\bigl[f'(W_N)\,\bigl|\,N\bigr]\Bigr]}\notag\\
&\leq \frac{\alpha a^2}{\sigma^2}\sqrt{\Var\bigl(E[D\,|\,N]\bigr)}\sqrt{E\Bigl[\bigl(E\bigl[f'(W_N)\,\bigl|\,N\bigr]\bigr)^2\Bigr]}\notag\\
&\leq\frac{\alpha a^2}{\sigma^2}\fnorm{f'}\sqrt{\Var\bigl(E[D\,|\,N]\bigr)}\,.
\end{align}
Now we will proceed by first assuming that $h$ is a $1$-Lipschitz function. In this case, we choose the coupling $(M,N)$ used for $E_1$ in such a way that $M\geq N$. By Remark \ref{sbrem} (a) such a construction of $(M,N)$ is always possible
e.g. via the quantile transformation und that it achieves the Wasserstein distance, i.e. 
\begin{equation*}
E\abs{M-N}=E[M-N]=\frac{E[N^2]}{E[N]}-E[N]=\frac{\Var(N)}{E[N]}=\frac{\gamma^2}{\alpha}=d_\W(N,N^s)\,. 
\end{equation*} 
In order to bound $E_1$, we first derive an estimate for $E\abs{W_M-W_N}$. We have 
\begin{align}\label{boundwscond}
 E\bigl[\abs{W_M-W_N}\,\bigl|\,N,M\bigr]&=\frac{1}{\sigma}E\bigl[\abs{S_{M}-S_N}\,\bigl|\,N,M\bigr]\leq\frac{\abs{M-N}}{\sigma}E\abs{X_1}\notag\\
 &\leq\frac{b(M-N)}{\sigma}
\end{align}
and, hence,
\begin{align}\label{bounddws}
 E\abs{W_M-W_N}&=E\Bigl[E\bigl[\abs{W_M-W_N}\,\bigl|\,N,M\bigr]\Bigr]=\frac{1}{\sigma}E\bigl[\abs{S_{M}-S_N}\,\bigl|\,N,M\bigr]\notag\\
 &\leq\frac{b}{\sigma}E[M-N]= \frac{b\gamma^2}{\sigma\alpha}\,.
\end{align}
Then, using \eqref{bounds}, \eqref{bounddws} as well as the fact that the $X_j$ are i.i.d., for $E_1$ we obtain that
\begin{align}
\abs{E_1}&=\frac{c^2\alpha}{\sigma^2}\Bigl|E\Bigl[f'(W_N)-f'\Bigl(W_M+\frac{Y-X_1}{\sigma}\Bigr)\Bigr]\Bigr|\notag\\
&\leq\frac{2c^2\alpha}{\sigma^2}\Bigl(E\abs{W_N-W_M}+\sigma^{-1}E\abs{Y-X_1}\Bigr)\label{e1gen}\\
&\leq\frac{2c^2\alpha}{\sigma^3}\Bigl(\frac{b\gamma^2}{\alpha}+\frac{3}{2c^2}E\babs{X_1-E[X_1]}^3\Bigr)\notag\\
&=\frac{2c^2b\gamma^2}{\sigma^3}+\frac{3\alpha d^3}{\sigma^3}\label{e1w}\,.
\end{align}
Here, we have used the inequality
\begin{equation}\label{zbdiff}
 E\abs{Y-X_1}=E\babs{Y-E[X_1]-\bigl(X_1-E[X_1]\bigr)}\leq\frac{3}{2\Var(X_1)}E\babs{X_1-E[X_1]}^3\,,
\end{equation}
which follows from an analogous one in the zero-bias framework (see \cite{CGS}) via the fact that $Y-E[X_1]$ has the $(X-E[X_1])$ - zero biased distribution.\\
Similarly to \eqref{boundwscond} we obtain 
\begin{equation*}
 E\bigl[\abs{V}\,\bigl|\,N,N^s\bigr]=E\bigl[\abs{W_{N^s}-W_N}\,\bigl|\,N,N^s\bigr]\leq\frac{b\abs{N^s-N}}{\sigma}=\frac{b\abs{D}}{\sigma}
\end{equation*}
which, together with \eqref{bounds} yields that
\begin{align}\label{e22w}
 \abs{E_{2,2}}&=\frac{\alpha a^2}{\sigma^2}\Babs{E\Bigl[(1-J)D\bigl(f'(W_N)-f'(W_{N^s})\bigr)\Bigr]}\leq\frac{2\alpha a^2}{\sigma^2}E\babs{(1-J)D(W_N-W_{N^s})}\notag\\
 &=\frac{2\alpha a^2}{\sigma^2}E\Bigl[(1-J)\abs{D}E\bigl[\abs{V}\,\bigl|\,N,N^s\bigr]\Bigr]\leq\frac{2\alpha a^2b}{\sigma^3}E\bigl[(1-J)D^2\bigr]\,.
\end{align}
We conclude the proof of the Wasserstein bounds by estimating $E_{2,3}$ and $E_{2,4}$. 
Note that by \eqref{fhtaylor} we have 
\[\babs{R_f(W_N,V)}\leq V^2\quad\text{and}\quad \babs{R_f(W_{N^s},-V)}\leq V^2\]
yielding 
\begin{align}\label{e234w}
\abs{E_{2,3}}+\abs{E_{2,4}}&\leq\frac{\alpha \abs{a}}{\sigma} E\Bigl[\bigl(1_{\{D\geq0\}}+1_{\{D<0\}}\bigr)V^2\Bigr]=\frac{\alpha \abs{a}}{\sigma}E\bigl[V^2\bigr]\,.
\end{align}
Observe that
\begin{align}\label{boundvsq1}
 E[V^2]&=\frac{1}{\sigma^2}E\Bigl[\bigl(S_{N^s}-S_{N}\bigr)^2\Bigr]\notag\\
 &=\frac{1}{\sigma^2}\Bigl(\Var\bigl(S_{N^s}-S_{N}\bigr)+\bigl(E\bigl[S_{N^s}-S_{N}\bigr]\bigr)^2\Bigr)
\end{align}
and
\begin{equation}\label{mpn8}
E\bigl[S_{N^s}-S_{N}\bigr]=E\Bigl[E\bigl[S_{N^s}-S_{N}\,\bigl|\,N,N^s\bigr]\Bigr]=aE[D]=\frac{a\gamma^2}{\alpha}\,.
\end{equation}
Further, from the variance decomposition formula we obtain
\begin{align*}
\Var\bigl(S_{N^s}-S_N\bigr)&=E\bigl[\Var\bigl(S_{N^s}-S_N\,\bigl|\,N,N^s\bigr)\bigr]+\Var\bigl(E\bigl[S_{N^s}-S_N,\bigl|\,N,N^s\bigr]\bigr)\notag\\
&=E\bigl[c^2\abs{D}\bigr]+\Var(aD)=c^2E\abs{D}+a^2\Var(D)\,.
\end{align*}
This together with \eqref{boundvsq1} and \eqref{mpn8} yields the bounds
\begin{align}
E[V^2]&=E\bigl[(W_{N^s}-W_N)^2\bigr]=\frac{1}{\sigma^2}\Bigl(c^2E\abs{D}+a^2E[D^2]\Bigr)\label{formvsq}\\
&\leq \frac{b^2}{\sigma^2}E[D^2]\label{boundvsq}\,,
\end{align}
where we have used the fact that $D^2\geq\abs{D}$ and $a^2+c^2=b^2$ to obtain
\[c^2E\abs{D}+a^2E[D^2]\leq b^2 E[D^2]\,.\]
The asserted bound on the Wasserstein distance between $W$ and $Z$ from Theorem \ref{maintheo} now follows from 
\eqref{bounds}, \eqref{esdec}, \eqref{e2}, \eqref{e1w}, \eqref{e22w}, \eqref{e234w} and \eqref{boundvsq}.\\
If $a=0$, then $E_1$ can be bounded more accurately than we did before. 
Indeed, using \eqref{formvsq} with $N^s=M$ and applying the Cauchy-Schwarz inequality give
\begin{align*}
 E\abs{W_M-W_N}&\leq \sqrt{E\bigl[(W_{M}-W_N)^2\bigr]}=\frac{c}{\sigma}\sqrt{E[M-N]}=\frac{c\gamma}{\sqrt{\alpha}\sigma}\,,
\end{align*}
as $c=b$ in this case. Plugging this into \eqref{e1gen}, we obtain  
\begin{align*}
\abs{E_1}&\leq \frac{2c^2\alpha}{\sigma^2}\Bigl(E\abs{W_M-W_N}+\sigma^{-1}E\abs{Y-X_1}\Bigr)\\
&\leq\frac{2c^3\gamma\sqrt{\alpha}}{\sigma^3}+ \frac{2c^2\alpha}{\sigma^3}E\abs{Y-X_1}\\
&\leq\frac{2c^3\gamma\sqrt{\alpha}}{c^3\alpha^{3/2}}+\frac{3\alpha d^3}{c^3\alpha^{3/2}}\\
&=\frac{2\gamma}{\alpha}+\frac{3 d^3}{c^3\sqrt{\alpha}}\,,
\end{align*}
which is the Wasserstein bound claimed in Theorem \ref{meanzero}.\\\\
Next, we proceed to the proof of the Berry-Esseen bounds in Theorems \ref{maintheo} and \ref{meanzero}.
Bounding the quantities $E_1$, $E_{2,2}$, $E_{2,3}$ and $E_{2,4}$ in the case that $h=h_z$ is much more technically involved. 
Also, in this case we do not in general profit from choosing $M$ appearing in $T_1$ in such a way that $M\geq N$.
This is why we let $M=N^s$ for the proof of the Kolmogorov bound in Theorem \ref{maintheo}. Only for the proof of Theorem \ref{meanzero} we will later assume that $M\geq N$. 
We write $f=f_z$ and introduce the notation
\begin{equation*}
 \Vtilde:=W^*-W=W_{N^s}+\sigma^{-1}(Y-X_1)-W_N=V+\sigma^{-1}(Y-X_1)\,.
\end{equation*}
From \eqref{e1} and the fact that $f$ solves the Stein equation \eqref{steineq} for $h=h_z$ we have 
\begin{align}\label{e1dec}
 E_1&=\frac{c^2\alpha}{\sigma^2}E\bigl[f'(W)-f'(W^*)\bigr]\notag\\
 &=\frac{c^2\alpha}{\sigma^2}E\bigl[Wf(W)-W^*f(W^*)\bigr]+\frac{c^2\alpha}{\sigma^2}\bigl(P(W\leq z)-P(W^*\leq z)\bigr)\notag\\
 &=:E_{1,1}+E_{1,2}\,.
\end{align}
In order to bound $E_{1,1}$ we apply \eqref{fzdiff} to obtain
\begin{equation}\label{e11}
 \abs{E_{1,1}}\leq\frac{c^2\alpha}{\sigma^2}E\Bigl[\abs{\Vtilde}\Bigl(\frac{\sqrt{2\pi}}{4}+\abs{W}\Bigr)\Bigr]\,.
\end{equation}
Using \eqref{formvsq}, \eqref{boundvsq} and \eqref{zbdiff} we have
\begin{align}
E\abs{\Vtilde}&\leq E\abs{V}+\sigma^{-1}E\abs{Y-X_1}\leq\sqrt{E[V^2]}+\frac{3 d^3}{2\sigma c^2}\notag\\
&=\frac{1}{\sigma}\sqrt{c^2E\abs{D}+a^2E[D^2]}+\frac{3 d^3}{2\sigma c^2}\label{e1z1}\\
&\leq\frac{b}{\sigma}\sqrt{E[D^2]}+\frac{3 d^3}{2\sigma c^2}\,.\label{e1z2}
\end{align}
Furthermore, using independence of $W$ and $Y$, we have 
\begin{align}\label{e1z3}
E\babs{(Y-X_1)W}&\leq E\babs{(Y-E[X_1])W}+E\abs{(X_1-E[X_1])W}\notag\\
&=E\babs{Y-E[X_1]}E\abs{W}+E\babs{(X_1-E[X_1])W}\notag\\
&\leq\frac{ d^3}{2c^2}\sqrt{E[W^2]}+\sqrt{\Var(X_1)E[W^2]}=\frac{ d^3}{2c^2} +c\,.
\end{align}
Finally, we have
\begin{align}
E\abs{VW}&\leq\sqrt{E[V^2]}\sqrt{E[W^2]}=\frac{1}{\sigma}\sqrt{c^2E\abs{D}+a^2E[D^2]}\label{e1z4}\\
&\leq\frac{b}{\sigma}\sqrt{E[D^2]}\label{e1z5}\,.
\end{align}
From \eqref{e11}, \eqref{e1z1}, \eqref{e1z2}, \eqref{e1z3}, \eqref{e1z4} and \eqref{e1z5} we conclude that 
\begin{align}
\abs{E_{1,1}}&\leq\frac{c^2\alpha}{\sigma^2}\Bigl(\frac{\sqrt{2\pi}}{4\sigma}\sqrt{c^2E\abs{D}+a^2E[D^2]}+\frac{3 d^3\sqrt{2\pi}}{8c^2\sigma}+\frac{ d^3}{2c^2\sigma} +\frac{c}{\sigma}\notag\\
&\;+\frac{1}{\sigma}\sqrt{c^2E\abs{D}+a^2E[D^2]}\Bigr)\notag\\
&=\frac{c^2\alpha(\sqrt{2\pi}+4)}{4\sigma^3}\sqrt{c^2E\abs{D}+a^2E[D^2]}+\frac{ d^3\alpha(3\sqrt{2\pi}+4)}{8\sigma^3}+ \frac{c^3\alpha}{\sigma^3}\label{e1z6}\\
&\leq\frac{(\sqrt{2\pi}+4)bc^2\alpha}{4\sigma^3}\sqrt{E[D^2]}+\frac{ d^3\alpha(3\sqrt{2\pi}+4)}{8\sigma^3}+ \frac{c^3\alpha}{\sigma^3}=:B_1\,.\label{e1z7}
\end{align}
In order to bound $E_{1,2}$ we need the following lemma, which will be proved in Section \ref{Appendix}.
In the following we denote by $C_\K$ the Berry-Esseen constant for sums of i.i.d. random variables with finite third moment. It is known from \cite{Shev11} that
\begin{equation*}
 C_\K\leq 0.4748\,.
\end{equation*}
In particular, $2C_\K\leq1$, which is substituted for $2C_\K$ in the statements of Theorems \ref{maintheo} and \ref{meanzero}. However, we prefer keeping the dependence of the bounds on $C_\K$ explicit within the proof.
\begin{lemma}\label{ce2}
 With the above assumptions and notation we have for all $z\in\R$
\begin{align}
 \babs{P(W^*\leq z)-P(W_{N^s}\leq z)}&\leq\frac{1}{\sqrt{\alpha}}\Bigl(\frac{7}{2}\sqrt{2}+2\Bigr)\frac{ d^3}{c^3}\label{cers1}\quad\text{and}\\
 \babs{P(W_{N^s}\leq z)-P(W\leq z)}&\leq P(N=0)+\frac{b}{c\sqrt{2\pi}}E\bigl[D1_{\{D\geq0\}}N^{-1/2}1_{\{N\geq1\}}\bigr]\notag\\
&\; +\frac{2C_\K d^3}{c^3} E\bigl[1_{\{D\geq0\}}N^{-1/2}1_{\{N\geq1\}}\bigr]\notag\\
&\;+\frac{1}{\sqrt{\alpha}}\Bigl(\frac{b}{c\sqrt{2\pi}}\sqrt{E\bigl[D^21_{\{D<0\}}\bigr]}+\frac{2C_\K d^3}{c^3}\sqrt{P(D<0)}\Bigr)\,.\label{cers2}
\end{align}
If $a=0$ and $D\geq0$, then for all $z\in\R$
\begin{align}
\babs{P(W_{N^s}\leq z)-P(W\leq z)}&\leq P(N=0)+\frac{2C_\K d^3}{c^3} E\bigl[N^{-1/2}1_{\{N\geq1\}}\bigr]\notag\\
&\;+\frac{1}{\sqrt{2\pi}}E\bigl[\sqrt{D}N^{-1/2}1_{\{N\geq1\}}\bigr]\label{cers3}\\
&\leq P(N=0)+\biggl(\frac{2C_\K d^3}{c^3}+\frac{\gamma}{\sqrt{\alpha}\sqrt{2\pi}}\biggr) \sqrt{E\bigl[1_{\{N\geq1\}}N^{-1}\bigr]}\label{cers4}\,.
\end{align}
\end{lemma}

Applying the triangle inequality to Lemma \ref{ce2} yields the following bounds on $E_{1,2}$: In the most general situation (Theorem \ref{maintheo} and Remark \ref{mtrem} (b)) we have
\begin{align}\label{e12gen}
 \abs{E_{1,2}}&\leq\Bigl(\frac{7}{2}\sqrt{2}+2\Bigr)\frac{\sqrt{\alpha} d^3}{c\sigma^2}+\frac{c^2\alpha}{\sigma^2}P(N=0)+\frac{\alpha bc}{\sigma^2\sqrt{2\pi}}E\bigl[D1_{\{D\geq0\}}N^{-1/2}1_{\{N\geq1\}}\bigr]\notag\\
 &\;+\frac{2C_\K d^3\alpha}{c\sigma^2}E\bigl[1_{\{D\geq0\}}N^{-1/2}1_{\{N\geq1\}}\bigr]+\frac{cb\sqrt{\alpha}}{\sigma^2\sqrt{2\pi}}\sqrt{E\bigl[D^21_{\{D<0\}}\bigr]}\notag\\
&\; +\frac{\alpha C_\K d^3}{c\sigma^2}\sqrt{P(D<0)}=:B_2\,.
\end{align}
If $a=0$ and $D\geq0$, then, keeping in mind that $\sigma^2=\alpha c^2$ in this case, 
\begin{align}
\abs{E_{1,2}}&\leq\Bigl(\frac{7}{2}\sqrt{2}+2\Bigr)\frac{ d^3}{c^3\sqrt{\alpha}}+P(N=0)+ \frac{2C_\K d^3}{c^3}E\bigl[N^{-1/2}1_{\{N\geq1\}}\bigr]\notag\\
&\;+\frac{1}{\sqrt{2\pi}}E\bigl[\sqrt{D}N^{-1/2}1_{\{N\geq1\}}\bigr]\notag\\
&\leq \Bigl(\frac{7}{2}\sqrt{2}+2\Bigr)\frac{ d^3}{c^3\sqrt{\alpha}}+P(N=0)
+\biggl(\frac{2C_\K d^3}{c^3}+\frac{\gamma}{\sqrt{\alpha}\sqrt{2\pi}}\biggr)\sqrt{E\bigl[1_{\{N\geq1\}}N^{-1}\bigr]}\,.\label{e12mz2}
\end{align}
The following lemma, which is also proved in Section \ref{Appendix}, will be needed to bound the quantities $E_{2,2}$, $E_{2,3}$ and $E_{2,4}$ from \eqref{e2}.
\begin{lemma}\label{auxlemma2}
With the above assumptions and notation we have 
\begin{align}
&E\bigl[J\abs{V}1_{\{z-(V\vee0)<W\leq z-(V\wedge0)\}}\bigr]\leq\frac{b}{\sigma}\sqrt{P(N=0)}\sqrt{E[JD^2]}\notag\\
&\;+\frac{b^2}{c\sigma\sqrt{2\pi}}E\bigl[JD^21_{\{N\geq1\}}N^{-1/2}\bigr]+\frac{2C_\K d^3 b}{c^3\sigma} E\bigl[JD1_{\{N\geq1\}}N^{-1/2}\bigr]\,,\label{aux21}\\
&E\bigl[(1-J)\abs{V}1_{\{z+(V\wedge0)<W_{N^s}\leq z+(V\vee0)\}}\bigr]\leq\frac{b^2}{c\sigma\sqrt{2\pi}}E\bigl[(1-J)D^2(N^s)^{-1/2}\bigr]\notag\\
&\;+\frac{2C_\K b d^3}{c^3\sigma\sqrt{\alpha}}\sqrt{E\bigl[(1-J)D^2\bigr]}\quad\text{and}\label{aux22}\\
&E\bigl[(1-J)\abs{D}1_{\{z+(V\wedge0)<W_{N^s}\leq z+(V\vee0)\}}\bigr]\leq\frac{b}{c\sqrt{2\pi}}E\bigl[(1-J)D^2(N^s)^{-1/2}\bigr]\notag\\
&\;+\frac{2C_\K  d^3}{c^3\sqrt{\alpha}}\sqrt{E\bigl[(1-J)D^2\bigr]}\label{aux23} \,.
\end{align}
\end{lemma}

Next, we derive a bound on $E_{2,2}$. Since $f$ solves the Stein equation \eqref{steineq} for $h=h_z$ we have 
\begin{align}\label{e22}
E_{2,2}&=\frac{\alpha a^2}{\sigma^2}E\bigl[(1-J)D\bigl(W_Nf(W_N)-W_{N^s}f(W_{N^s})\bigr)\bigr]\notag\\
&\;+\frac{\alpha a^2}{\sigma^2}E\bigl[(1-J)D\bigl(1_{\{W_N\leq z\}}-1_{\{W_{N^s}\leq z\}}\bigr)\bigr]\notag\\
&=:E_{2,2,1}+E_{2,2,2}\,.
\end{align}
Using 
\[W_N=W_{N^s}-V\]
and Lemma \ref{indlemma}, we obtain from \eqref{aux23} that
\begin{align}\label{e222}
 \abs{E_{2,2,2}}&\leq\frac{\alpha a^2}{\sigma^2}E\bigl[1_{\{D<0\}}\abs{D}1_{\{z+(V\wedge0)<W_{N^s}\leq z+(V\vee0)\}}\bigr]\notag\\
 &\leq\frac{\alpha a^2b}{\sigma^2c\sqrt{2\pi}}E\bigl[1_{\{D<0\}}D^2(N^s)^{-1/2}\bigr]
+\frac{2C_\K  d^3 a^2\sqrt{\alpha}}{c^3\sigma^2}\sqrt{E\bigl[1_{\{D<0\}}D^2\bigr]}\notag\\
&=:B_4\,.
\end{align}
As to $E_{2,2,1}$, from \eqref{fzdiff} we have 
\begin{align}\label{e221}
 \abs{E_{2,2,1}}&\leq\frac{\alpha a^2}{\sigma^2}E\Bigl[(1-J)\abs{DV}\Bigl(\abs{W_{N^s}}+\frac{\sqrt{2\pi}}{4}\Bigr)\Bigr]
\end{align}
As
\begin{equation}\label{e22z1}
 E\bigl[\abs{V}\,\bigl|\,N,N^s\bigr]\leq\sqrt{E\bigl[V^2\,\bigl|\,N,N^s\bigr]}=\frac{1}{\sigma}\sqrt{c^2\abs{D}+a^2D^2}\leq\frac{b}{\sigma}\abs{D}\,,
\end{equation}
by conditioning, we see 
\begin{align}\label{e22z2}
 E\bigl[(1-J)\abs{DV}]&=E\Bigl[(1-J)\abs{D}E\bigl[\abs{V}\,\bigl|\,N,N^s\bigr]\Bigr]\leq \frac{b}{\sigma}E\bigl[(1-J)D^2]\,.
\end{align}
Now, using the fact that conditionally on $N^s$, the random variables $W_{N^s}$ and $(1-J)\abs{DV}$ are independent, as well as the Cauchy-Schwarz inequality, we conclude that 
\begin{align}\label{e22z3}
 E\babs{(1-J)DVW_{N^s}}&=E\Bigl[E\bigl[\babs{(1-J)DVW_{N^s}}\,\bigl|\,N^s\bigr]\Bigr]\notag\\
 &=E\Bigl[E\bigl[(1-J)\abs{DV}\,\bigl|\,N^s\bigr]E\bigl[\abs{W_{N^s}}\,\bigl|\,N^s\bigr]\Bigr]\notag\\
 &\leq\sqrt{E\Bigl[\bigl(E\bigl[(1-J)\abs{DV}\,\bigl|\,N^s\bigr]\bigr)^2\Bigr]}\sqrt{E\Bigl[\bigl(E\bigl[\abs{W_{N^s}}\,\bigl|\,N^s\bigr]\bigr)^2\Bigr]}\notag\\
 &\leq\frac{b}{\sigma}\sqrt{E\Bigl[\bigl(E\bigl[(1-J)D^2\,\bigl|\,N^s\bigr]\bigr)^2\Bigr]}\sqrt{E\Bigl[W_{N^s}^2\Bigr]}\,,
\end{align}
where we have used the conditional Jensen inequality, \eqref{e22z1} and 
\begin{equation*}
 E\bigl[(1-J)\abs{DV}\,\bigl|\,N^s\bigr]=E\Bigr[(1-J)\abs{D}E\bigl[\abs{V}\,\bigl|\,N,N^s\bigr]\,\Bigl|\,N^s\Bigr]
\end{equation*}
to obtain the last inequality. Using the defining relation \eqref{sbdef} of the size-biased distribution one can easily show that 
\begin{align}\label{secmomwns}
E\Bigl[W_{N^s}^2\Bigr]&=\frac{1}{\sigma^2}E\Bigl[c^2N^s+a^2(N^s-\alpha)^2\Bigr]
=\frac{c^2\beta^2+a^2\bigl(\delta^3-2\alpha\beta^2+\alpha^3\bigr)}{\alpha\sigma^2}\,,
\end{align}
which, together with \eqref{e221}, \eqref{e22z2} and \eqref{e22z3} yields that 
\begin{align}\label{e221b}
\abs{E_{2,2,1}}&\leq\frac{\alpha a^2b\sqrt{2\pi}}{4\sigma^3}E\bigl[1_{\{D<0\}}D^2] \notag\\
&\;+a^2b\frac{c^2\beta^2+a^2\bigl(\delta^3-2\alpha\beta^2+\alpha^3\bigr)}{\sigma^5}\sqrt{E\Bigl[\bigl(E\bigl[1_{\{D<0\}}D^2\,\bigl|\,N^s\bigr]\bigr)^2\Bigr]}\notag\\
&=:B_5\,.
\end{align}

It remains to bound the quantities $E_{2,3}$ and $E_{2,4}$ from \eqref{e2} for $f=f_z$. From \eqref{fztaylor} 
we have 
\begin{align}\label{r1}
\abs{E_{2,3}}&=\frac{\alpha \abs{a}}{\sigma}\Babs{E\bigl[1_{\{D\geq0\}}R_f(W,V)\bigr]}\notag\\
&\leq\frac{\alpha \abs{a}}{2\sigma}E\Bigl[JV^2\Bigl(\abs{W}+\frac{\sqrt{2\pi}}{4}\Bigr)\Bigr]\notag\\
&\;+\frac{\alpha \abs{a}}{\sigma}E\Bigl[J\abs{V}1_{\{z-(V\vee0)<W\leq z-(V\wedge0)\}}\Bigr]=:R_{1,1}+R_{1,2}\,.
\end{align}
Similarly to \eqref{formvsq} we obtain 
\begin{align}\label{boundjvsq}
 E\bigl[JV^2\bigr]&=\frac{1}{\sigma^2}\Bigl(c^2E\bigl[JD\bigr]+a^2E\bigl[JD^2\bigr]\Bigr)\notag\\
 &\leq\frac{b^2}{\sigma^2}E\bigl[JD^2\bigr]
\end{align}
from
\begin{align}\label{r1e1}
 E\bigl[JV^2\,\bigl|\,N,N^s\bigr]&=JE\bigl[V^2\,\bigl|\,N,N^s\bigr]=\frac{J}{\sigma^2}\Bigl(c^2\abs{D}+a^2D^2\Bigr)\notag\\
 &=\frac{1}{\sigma^2}\Bigl(c^2 JD+a^2JD^2\Bigr)\,.
 \end{align}
Also, recall that the random variables 
\[JV^2=\sigma^{-1}1_{\{N^s\geq N\}}\Bigl(\sum_{j=N+1}^{N^s}X_j\Bigr)^2\quad\text{and}\quad W_N=\sigma^{-1}\Bigl(\sum_{j=1}^N X_j-\alpha a\Bigr)\]
are conditionally independent given $N$. Hence, using the Cauchy-Schwarz inequality
\begin{align}\label{r1e2}
 E\bigl[JV^2\abs{W_N}\bigr]&=E\Bigl[E\bigl[JV^2\abs{W_N}\,\bigl|\,N\bigr]\Bigr]=E\Bigl[E\bigl[JV^2\,\bigl|\,N\bigr]E\bigl[\abs{W_N}\,\bigl|\,N\bigr]\Bigr]\notag\\
&\leq\sqrt{E\biggl[\Bigl(E\bigl[JV^2\,\bigl|\,N\bigr]\Bigr)^2\biggr]}\sqrt{E\biggl[\Bigl(E\bigl[\abs{W_N}\,\bigl|\,N\bigr]\Bigr)^2\biggr]}\,.
\end{align}
From \eqref{r1e1} and $D^2\geq\abs{D}$ we conclude that 
\begin{align}\label{r1e3}
 E\bigl[JV^2\,\bigl|\,N\bigr]=\frac{1}{\sigma^2}\Bigl(c^2E\bigl[JD\,\bigl|\,N\bigr]+a^2E\bigl[JD^2\,\bigl|\,N\bigr]\Bigr)\leq\frac{b^2}{\sigma^2}E\bigl[JD^2\,\bigl|\,N\bigr]\,.
\end{align}
Furthermore, by the conditional version of Jensen' s inequality 
we have
\begin{align}\label{r1e4}
 E\biggl[\Bigl(E\bigl[\abs{W_N}\,\bigl|\,N\bigr]\Bigr)^2\biggr]&\leq E\Bigl[E\bigl[W_N^2\,\bigl|\,N\bigr]\Bigr]
=E\bigl[W_N^2\bigr]=1\,.
\end{align}
Thus, from \eqref{r1e2}, \eqref{r1e3} and \eqref{r1e4} we see that 
\begin{equation}\label{r1e5}
 E\bigl[JV^2\abs{W_N}\bigr]\leq\frac{b^2}{\sigma^2}\sqrt{E\Bigl[\bigl(E\bigl[JD^2\,\bigl|\,N\bigr]\bigr)^2\Bigr]}\,.
\end{equation}
Hence, \eqref{r1}, \eqref{boundjvsq} and \eqref{r1e5} yield 
\begin{align}\label{r11}
 R_{1,1}&\leq\frac{\alpha \abs{a}b^2}{2\sigma^3}\sqrt{E\Bigl[\bigl(E\bigl[JD^2\,\bigl|\,N\bigr]\bigr)^2\Bigr]}
+\frac{\alpha \abs{a}b^2\sqrt{2\pi}}{8\sigma^3}E[JD^2]\,.
\end{align}
Finally, from \eqref{r1}, \eqref{r11} and \eqref{aux21} we get
\begin{align}\label{boundr1}
 \abs{E_{2,3}}&\leq\frac{\alpha \abs{a}b^2}{2\sigma^3}\sqrt{E\Bigl[\bigl(E\bigl[1_{\{D\geq0\}}D^2\,\bigl|\,N\bigr]\bigr)^2\Bigr]}
+\frac{\alpha \abs{a}b^2\sqrt{2\pi}}{8\sigma^3}E[1_{\{D\geq0\}}D^2]\notag\\
 &\;+\frac{\alpha \abs{a}b}{\sigma^2}\sqrt{P(N=0)}\sqrt{E[1_{\{D\geq0\}}D^2]}\notag\\
&\;+\frac{\alpha \abs{a}b^2}{c\sigma^2\sqrt{2\pi}}E\bigl[1_{\{D\geq0\}}D^21_{\{N\geq1\}}N^{-1/2}\bigr]+\frac{2C_\K d^3\alpha \abs{a} b}{\sigma^2} E\bigl[1_{\{D\geq0\}}D1_{\{N\geq1\}}N^{-1/2}\bigr]\notag\\
&=:B_6\,.
\end{align}
Similarly, we have
\begin{align}\label{r2}
\abs{E_{2,4}}&=\frac{\alpha \abs{a}}{\sigma}\Babs{E\bigl[1_{\{D<0\}}R_f(W_{N^s},-V)\bigr]}\notag\\
&\leq\frac{\alpha \abs{a}}{2\sigma}E\Bigl[(1-J)V^2\Bigl(W_{N^s}+\frac{\sqrt{2\pi}}{4}\Bigr)\Bigr]\notag\\
&\;+\frac{\alpha \abs{a}}{\sigma}E\Bigl[(1-J)\abs{V}1_{\{z+(V\wedge0)<W_{N^s}\leq z+(V\vee0)\}}\Bigr]=:R_{2,1}+R_{2,2}\,.
\end{align}
Analogously to the above we obtain
\begin{align}
 E\bigl[(1-J)V^2\bigr]&=\frac{1}{\sigma^2}\Bigl(c^2E\bigl[(1-J)\abs{D}\bigr]+a^2E\bigl[(1-J)D^2\bigr]\Bigr)\notag\\
 &\leq\frac{b^2}{\sigma^2}E\bigl[(1-J)D^2\bigr]\quad\text{and}\label{boundnjvsq}\\
E\bigl[(1-J)V^2\,\bigl|\,N^s\bigr]&=\frac{1}{\sigma^2}\Bigl(c^2E\bigl[(1-J)\abs{D}\,\bigl|\,N^s\bigr]+a^2E\bigl[(1-J)D^2\,\bigl|\,N^s\bigr]\Bigr)\notag\\
&\leq\frac{b^2}{\sigma^2}E\bigl[(1-J)D^2\,\bigl|\,N^s\bigr]\notag\,.
\end{align}
Using these as well as the conditional independence of $(1-J)V^2$ and $W_{N^s}$ given $N^s$, one has
\begin{equation}\label{r21}
 E\bigl[(1-J)V^2\abs{W_{N^s}}\bigr]\leq\frac{b^2}{\sigma^2}\sqrt{E\Bigl[\bigl(E\bigl[(1-J)D^2\,\bigl|\,N^s\bigr]\bigr)^2\Bigr]}
\sqrt{E\Bigl[W_{N^s}^2\Bigr]}\,.
\end{equation}
Combining \eqref{secmomwns} and \eqref{r21} we obtain
\begin{align}\label{r22}
E\bigl[(1-J)V^2\abs{W_{N^s}}\bigr]&\leq\frac{b^2}{\sigma^2}\sqrt{E\Bigl[\bigl(E\bigl[(1-J)D^2\,\bigl|\,N^s\bigr]\bigr)^2\Bigr]}\notag\\
&\quad\left(\frac{c^2\beta^2+a^2\bigl(\delta^3-2\alpha\beta^2+\alpha^3\bigr)}{\alpha\sigma^2}\right)^{1/2}\,.
\end{align}
Thus, from \eqref{r2}, \eqref{boundnjvsq}, \eqref{r22} and \eqref{aux22} we conclude
\begin{align}\label{boundr2}
\abs{E_{2,4}}&\leq\frac{\alpha\abs{a}b^2\sqrt{2\pi}}{8\sigma^3} E\bigl[1_{\{D<0\}}D^2\bigr]+\frac{\alpha\abs{a}b^2}{2\sigma^3}\sqrt{E\Bigl[\bigl(E\bigl[1_{\{D<0\}}D^2\,\bigl|\,N^s\bigr]\bigr)^2\Bigr]}\notag\\
&\quad\cdot\left(\frac{c^2\beta^2+a^2\bigl(\delta^3-2\alpha\beta^2+\alpha^3\bigr)}{\alpha\sigma^2}\right)^{1/2}\notag\\
&\;+\frac{\alpha\abs{a}b^2}{\sigma^2c\sqrt{2\pi}}E\bigl[1_{\{D<0\}}D^2(N^s)^{-1/2}\bigr]+\frac{\sqrt{\alpha}\abs{a}2C_\K b d^3}{\sigma^2}\sqrt{E\bigl[1_{\{D<0\}}D^2\bigr]}\notag\\
&=:B_7\,.
\end{align}
The Berry-Esseen bound stated in Remark \ref{mtrem} (b) follows from \eqref{esdec}, \eqref{e1dec}, \eqref{e1z7}, \eqref{e12gen}, \eqref{e2}, \eqref{e21}, \eqref{boundfzp}, \eqref{e22}, \eqref{e222}, \eqref{e221b}, \eqref{boundr1} and 
\eqref{boundr2}. This immediately yields the Berry-Esseen bound presented in Theorem \ref{maintheo} (b) because 
\[B_2=B_4=B_5=B_7=0\]
in this case. In order to obtain the Kolmogorov bound in Theorem \ref{meanzero}, again, we choose $M$ such that $M\geq N$ and use the bounds \eqref{e1z6} and \eqref{e12mz2} instead. The result then follows 
from \eqref{esdec} and \eqref{e1dec}.

\section{Proofs of auxiliary results}\label{Appendix}
Here, we give several rather technical proofs.  We start with the following easy lemma, whose proof is omitted.
\begin{lemma}\label{indlemma}
 For all $x,u,v,z\in\R$ we have 
 \begin{align*}
 1_{\{x+u\leq z\}}-1_{\{x+v\leq z\}}&=1_{\{z-v<x\leq z-u\}}-1_{\{z-u<x\leq z-v\}}\quad\text{and}\\
 \babs{1_{\{x+u\leq z\}}-1_{\{x+v\leq z\}}}&=1_{\{z- u\vee v<x\leq z- u\wedge v\}}\,.
\end{align*}
 \end{lemma}

\begin{lemma}[Concentration inequality]\label{ce1}
For all real $t<u$ and for all $n\geq1$ we have
\begin{equation*}
 P(t<W_n\leq u)\leq \frac{\sigma(u-t)}{c\sqrt{2\pi}\sqrt{n}}+\frac{2 C_\K d^3}{c^3\sqrt{n}}\,.
\end{equation*}
\end{lemma}

\begin{proof}
The proof uses the Berry-Esseen Theorem for sums of i.i.d. random variables with finite third moment as well as the following fact, whose proof is straightforward: For each real-valued random variable $X$ and for all real $r<s$ we have the 
bound
\begin{equation}\label{ce1eq1}
 P(r<X\leq s)\leq\frac{s-r}{\sqrt{2\pi}}+2d_\K(X,Z)\,.
\end{equation}
A similar result was used in \cite{PekRol11} in the framework of exponential approximation. Now, for given $t<u$ and $n\geq 1$ by \eqref{ce1eq1} and the Berry-Esseen Theorem we have 
\begin{align*}
 P(t<W_n\leq u)&=P\biggl(\frac{\sigma t+a(\alpha-n)}{c\sqrt{n}}<\frac{S_n-na}{c\sqrt{n}}\leq \frac{\sigma u+a(\alpha-n)}{c\sqrt{n}}\biggr)\notag\\
 &\leq\frac{\sigma(u-t)}{c\sqrt{2\pi}\sqrt{n}}+\frac{2C_\K d^3}{c^3\sqrt{n}}\,.
\end{align*}
\end{proof}

\begin{remark}
It is actually not strictly necessary to apply the Berry-Esseen Theorem in order to prove Lemma \ref{ce1}: Using known concentration results for sums of independent random variables like 
Proposition 3.1 from \cite{CGS}, for instance, would yield a comparable result, albeit with worse constants. 
\end{remark}

In order to prove Lemma \ref{ce2} we cite the following concentration inequality from \cite{CGS}:
\begin{lemma}\label{cecgs}
Let $Y_1,\dotsc,Y_n$ be independent mean zero random variables such that 
\[\sum_{j=1}^n E[Y_j^2]=1\quad\text{and}\quad\zeta:=\sum_{j=1}^n E\abs{Y_j}^3<\infty\,,\]
then with $S^{(i)}:=\sum_{j\not= i}Y_j$ one has for all real $r<s$ and all $i=1,\dotsc,n$ that
\begin{equation*}
P(r\leq S^{(i)}\leq s)\leq\sqrt{2}(s-r)+2(\sqrt{2}+1)\zeta\,. 
\end{equation*}
\end{lemma}

\begin{proof}[Proof of Lemma \ref{ce2}]
We first prove \eqref{cers1}. Define 
\begin{equation*}
W_{N^s}^{(1)}:=W_{N^s}-\sigma^{-1}X_1=\frac{1}{\sigma}\Bigl(\sum_{j=2}^{N^s} X_j-\alpha a\Bigr) 
\end{equation*}
such that 
\begin{equation*}
 W_{N^s}=W_{N^s}^{(1)}+\sigma^{-1}X_1\quad\text{and}\quad W^*=W_{N^s}^{(1)}+\sigma^{-1}Y\,.
\end{equation*}
Then, using Lemma \ref{indlemma} we have 
\begin{align*}
&\:\babs{P(W^*\leq z)-P(W_{N^s}\leq z)}\notag\\
&=\babs{P(W_{N^s}^{(1)}+\sigma^{-1}Y\leq z)-P(W_{N^s}^{(1)}+\sigma^{-1}X_1\leq z)}\notag\\
&\leq P\bigl(z-\sigma^{-1} (X_1\vee Y)<W_{N^s}^{(1)}\leq z-\sigma^{-1} (X_1\wedge Y)\bigr)\notag\\
&=E\Biggl[P\biggl(\frac{\sigma z- (X_1\vee Y)+a(\alpha-N^s+1)}{c\sqrt{N^s}}<\sum_{j=2}^{N^s}\Bigl(\frac{X_j-a}{c\sqrt{N^s}}\Bigr)\notag\\ 
&\hspace{3cm}\leq \frac{\sigma z- (X_1\wedge Y)+a(\alpha-N^s+1)}{c\sqrt{N^s}}\,\biggl|\,N^s\biggr)\Biggr]\,.
\end{align*}
Now note that conditionally on $N^s$ the random variables $W_{N^s}^{(1)}$ and $(X_1,Y)$ are independent and that the statement of Lemma \ref{cecgs} may be applied to the random variable in the middle term of the above conditional probabilty 
giving the bound 
\begin{align}\label{ce22}
&\:\babs{P(W^*\leq z)-P(W_{N^s}\leq z)}
\leq E\biggl[\frac{\sqrt{2}\abs{Y-X_1}}{c\sqrt{N^s}}+\frac{2(\sqrt{2}+1) d^3}{c^3\sqrt{N^s}}\biggr]\,.
\end{align}
Noting that $(X_1,Y)$ and $N^s$ are independent and using \eqref{zbdiff} again, we obtain 
\begin{equation}\label{ce23}
 E\biggl[\frac{\abs{Y-X_1}}{\sqrt{N^s}}\biggr]\leq\frac{3}{2c^2} d^3 E\bigl[(N^s)^{-1/2}\bigr]\leq \frac{3 d^3}{2c^2\sqrt{\alpha}}\,, 
\end{equation}
as
\begin{equation}\label{nsexp1}
 E\bigl[(N^s)^{-1/2}\bigr]=\frac{E[\sqrt{N}]}{E[N]}\leq\frac{\sqrt{E[N]}}{E[N]}=\frac{1}{\sqrt{\alpha}}
\end{equation}
by \eqref{sbdef} and Jensen's inequality. From \eqref{ce22}, \eqref{ce23} and \eqref{nsexp1} the bound \eqref{cers1} follows.\\
Next we prove \eqref{cers2}.
Using Lemma \ref{indlemma} we obtain 
\begin{align}\label{ce24}
\babs{P(W_{N^s}\leq z)-P(W\leq z)}&=\babs{E\bigl[J(1_{\{W+V\leq z\}}-1_{\{W\leq z\}})\bigr]\notag\\
&\;-E\bigl[(1-J)(1_{\{W_{N^s}-V\leq z\}}-1_{\{W_{N^s}\leq z\}})\bigr]}\notag\\
&\leq E\bigl[J1_{\{z-(V\vee0)<W\leq z-(V\wedge0)\}}\bigr]\notag\\
&\;+E\bigl[(1-J)1_{\{z+(V\wedge0)<W_{N^s}\leq z+(V\vee0)\}}\bigr]\notag\\
&=:A_1+A_2\,.
\end{align}
To bound $A_1$ we write 
\begin{align}\label{ce25}
 A_1&=\sum_{n=0}^\infty E\bigl[J1_{\{N=n\}}1_{\{z-(V\vee0)<W\leq z-(V\wedge0)\}}\bigr]\notag\\
 &=\sum_{n=0}^\infty P\bigl(z-(V\vee0)<W\leq z-(V\wedge0)\,\bigl|\,D\geq0,N=n\bigr)\cdot P\bigl(D\geq0,N=n\bigr)\,.
\end{align}
Now note that conditionally on the event that $D\geq0$ and $N=n$ the random variables $W$ and $V$ are independent and 
\begin{equation*}
 \calL(W\,|\,D\geq0,N=n)=\calL(W_n)\,.
\end{equation*}
Thus, using Lemma \ref{ce1} we have for all $n\geq1$:
\begin{align}\label{ce26}
 &\:P\bigl(z-(V\vee0)<W\leq z-(V\wedge0)\,\bigl|\,D\geq0,N=n\bigr)\notag\\
 &=P\bigl(z-(V\vee0)<W_n\leq z-(V\wedge0)\,\bigl|\,D\geq0,N=n\bigr)\notag\\
 &\leq E\biggl[\frac{\sigma\abs{V}}{c\sqrt{2\pi}\sqrt{n}}+\frac{2C_\K d^3}{c^3\sqrt{n}}\,\biggl|\,D\geq0,N=n\biggr]
\end{align}
From \eqref{ce25} and \eqref{ce26} we thus have 
\begin{align}\label{ce27}
 A_1&\leq P(N=0)+\sum_{n=1}^\infty E\biggl[\frac{\sigma\abs{V}}{c\sqrt{2\pi}\sqrt{n}}+\frac{2C_\K d^3}{c^3\sqrt{n}}\,\biggl|\,D\geq0,N=n\biggr]P\bigl(D\geq0,N=n\bigr)\notag\\
 &= P(N=0)+\sum_{n=1}^\infty E\biggl[1_{\{D\geq0,N=n\}}\biggl(\frac{\sigma\abs{V}}{c\sqrt{2\pi}\sqrt{n}}+\frac{2C_\K d^3}{c^3\sqrt{n}}\biggr)\biggr]\notag\\
 &=P(N=0)+E\biggl[J1_{\{N\geq1\}}\biggl(\frac{\sigma\abs{V}}{c\sqrt{2\pi}\sqrt{N}}+\frac{2C_\K d^3}{c^3\sqrt{N}}\biggr)\biggr]
\end{align}
Now note that 
\begin{align}
 E\bigl[J1_{\{N\geq1\}}\abs{V}N^{-1/2}\bigr]&=E\Bigl[J1_{\{N\geq1\}}N^{-1/2}E\bigl[\abs{V}\,\bigl|\,N,N^s\bigr]\Bigr]\notag\\
 &\leq E\Bigl[J1_{\{N\geq1\}}N^{-1/2}\sqrt{E\bigl[V^2\,\bigl|\,N,N^s\bigr]}\Bigr]\notag\\
 &=\frac{1}{\sigma}E\Bigl[J1_{\{N\geq1\}}N^{-1/2}\sqrt{c^2D+a^2D^2}\Bigr]\label{ce28mz}\\
 &\leq\frac{b}{\sigma}E\bigl[JD1_{\{N\geq1\}}N^{-1/2}\bigr]\label{ce28}\,.
\end{align}
It remains to bound $A_2$. We may assume that $P(D<0)>0$ since otherwise $A_2=0$. Noting that $N^s\geq1$ almost surely, similarly to \eqref{ce25} we obtain
\begin{align*}
A_2&=\sum_{m=1}^\infty P\bigl(z+(V\wedge0)<W_{N^s}\leq z+(V\vee0)\,\bigl|\,D<0,N^s=m\bigr)\notag\\
&\qquad\cdot P\bigl(D<0,N^s=m\bigr)\,.
\end{align*}
Now, using the fact that conditionally on the event $\{N^s=m\}\cap\{D<0\}$ the random variables $W_{N^s}$ and $V$ are independent and 
\begin{equation*}
\calL(W_{N^s}|N^s=m,D<0)=\calL(W_m)
\end{equation*}
in the same manner as \eqref{ce27} we find
\begin{equation}\label{ce211}
A_2\leq E\biggl[(1-J)\biggl(\frac{\sigma\abs{V}}{c\sqrt{2\pi}\sqrt{N^s}}+\frac{2C_\K d^3}{c^3\sqrt{N^s}}\biggr)\biggr]\,.
\end{equation}
Using \eqref{sbdef} we have 
\begin{equation}\label{nsexp2}
E\bigl[\bigl(N^s\bigr)^{-1}\bigr]=\frac{1}{E[N]}=\frac{1}{\alpha}\,.
\end{equation}

Thus, from the Cauchy-Schwarz inequality and \eqref{nsexp2} we obtain
\begin{align}
E\Bigl[(1-J)\frac{\abs{V}}{\sqrt{N^s}}\Bigr]&\leq\sqrt{E\bigl[\bigl(N^s\bigr)^{-1}\bigr]}\sqrt{E\bigl[(1-J)V^2\bigr]}\notag\\
&=\frac{1}{\sigma\sqrt{\alpha}}\sqrt{c^2 E\bigl[\abs{D}(1-J)\bigr]+a^2E\bigl[D^2(1-J)\bigr]}\notag\\
&\leq\frac{b}{\sigma\sqrt{\alpha}}\sqrt{E\bigl[D^2(1-J)\bigr]}\label{ce213}\,.
\end{align}
Similarly, we have
\begin{align}\label{ce14}
E\Bigl[\frac{1-J}{\sqrt{N^s}}\Bigr]&\leq\sqrt{P(D<0)}\sqrt{E\bigl[\bigl(N^s\bigr)^{-1}\bigr]}
=\frac{\sqrt{P(D<0)}}{\sqrt{\alpha}}\,.
\end{align}
Thus, from \eqref{ce27}, \eqref{ce28}, \eqref{ce211}, \eqref{ce213} and \eqref{ce14} we see that $A_1+A_2$ is bounded from above 
by the right hand side of \eqref{cers2}. Using \eqref{ce27} and \eqref{ce28mz} instead gives the bounds \eqref{cers3} and \eqref{cers4}. \\
\end{proof}

\begin{proof}[Proof of Lemma \ref{auxlemma2}]
We only prove \eqref{aux21}, the proofs of \eqref{aux22} and \eqref{aux23} being similar and easier. 
By the definition of conditional expectation given an event, we have 
\begin{align}\label{al1}
 &\quad E\bigl[J\abs{V}1_{\{z-(V\vee0)<W_N\leq z-(V\wedge0)\}}\bigr]\notag\\
 &=\sum_{n=0}^\infty E\bigl[1_{\{N=n,D\geq0\}}\abs{V}1_{\{z-(V\vee0)<W_n\leq z-(V\wedge0)\}}\bigr]=E\bigl[1_{\{N=0\}}J\abs{V}\bigl]\notag\\
 &\quad+\sum_{n=1}^\infty E\bigl[\abs{V}1_{\{z-(V\vee0)<W_n\leq z-(V\wedge0)\}}\,\bigl|\,N=n,D\geq0\bigr]\cdot P(N=n,D\geq0)\,.
\end{align}
Now, for $n\geq1$, using the fact that the random variables $W_N$ and $V$ are conditionally independent given the event $\{D\geq0\}\cap\{N=n\}$, from Lemma \ref{ce1} we infer that
\begin{align}\label{al2}
 &\quad E\bigl[\abs{V}1_{\{z-(V\vee0)<W_n\leq z-(V\wedge0)\}}\,\bigl|\,N=n,D\geq0\bigr]\notag\\
 &=E\Bigl[\abs{V}\Bigl(\frac{\sigma\abs{V}}{c\sqrt{2\pi}\sqrt{n}}+\frac{2 C_\K d^3}{c^3\sqrt{n}}\Bigr)\,\Bigl|\,N=n,D\geq0\Bigr]
\end{align}
Combining \eqref{al1} and \eqref{al2} we get
\begin{align}\label{al3}
 &\quad E\bigl[J\abs{V}1_{\{z-(V\vee0)<W_N\leq z-(V\wedge0)\}}\bigr] \leq E\bigl[1_{\{N=0\}}J\abs{V}\bigl]\notag\\
 &\quad+\sum_{n=1}^\infty E\Bigl[\abs{V}\Bigl(\frac{\sigma\abs{V}}{c\sqrt{2\pi}\sqrt{n}}+\frac{2 C_\K d^3}{c^3\sqrt{n}}\Bigr)\,\Bigl|\,N=n,D\geq0\Bigr]\cdot P(N=n,D\geq0)\notag\\
 &=E\bigl[1_{\{N=0\}}J\abs{V}\bigl]+E\Bigl[1_{\{N\geq1\}}J\abs{V}\Bigl(\frac{\sigma\abs{V}}{c\sqrt{2\pi}\sqrt{N}}+\frac{2 C_\K d^3}{c^3\sqrt{N}}\Bigr)\Bigr]\,.
\end{align}
Using Cauchy-Schwarz as well as
\begin{equation*}
 E\bigl[JV^2\bigr]=E\Bigl[JE\bigl[V^2\,\bigl|\,N,N^s\bigr]\Bigr]\leq\frac{b^2}{\sigma^2}E\bigl[JD^2\bigr]
\end{equation*}
we obtain 
\begin{equation}\label{al4}
 E\bigl[1_{\{N=0\}}J\abs{V}\bigl]\leq\frac{b}{\sigma}\sqrt{P(N=0)}\sqrt{E\bigl[JD^2\bigr]}\,.
\end{equation}
Analogously to \eqref{ce28} one can show that 
\begin{equation}\label{al5}
E\bigl[1_{\{N\geq1\}}N^{-1/2}JV^2\bigr]\leq \frac{b^2}{\sigma^2}E\bigl[1_{\{N\geq1\}}N^{-1/2}JD^2\bigr]\,.
\end{equation}
Hence, bound \eqref{aux21} follows from \eqref{al3}, \eqref{al4}, \eqref{ce28} and \eqref{al5}.\\ 
\end{proof}

\section*{Acknowledgements}
The author would like to thank an anonymous referee for useful comments and suggestions concerning the presentation of this work.

\normalem
\bibliography{rs15v1}{}
\bibliographystyle{alpha}

\end{document}